\documentclass[preprint,authoryear]{elsarticle}
\usepackage{epsfig}
\usepackage{amsmath}
\usepackage{amstext}
\usepackage{amsfonts}
\usepackage{amssymb}
\usepackage{eucal}
\usepackage{graphicx}
\usepackage{verbatim}
\usepackage{bm}
\usepackage{flushend}

\usepackage{tikz}
\usepackage{pgfplots}

\usetikzlibrary{arrows,petri,topaths}
\usepackage{tkz-berge}

\newcommand{\RR}{{\mathbb{R}}}
\newcommand{\NN}{{\mathbb{N}}}

\newcommand{\CC}{{\mathbb{C}}}

\newcommand{\trans}{{\sf T}}

\newcommand{\asto}{\overset{\rm a.s.}{\longrightarrow}}

\newcommand{\EE}{{\rm E}}

\DeclareMathOperator{\tr}{tr}
\DeclareMathOperator{\diag}{diag}
\DeclareMathOperator{\argmin}{argmin}

\newcounter{ctheorem}
\newtheorem{theorem}[ctheorem]{Theorem}
\newcounter{cassumption}
\newtheorem{assumption}[cassumption]{Assumption}

\newcounter{ccorollary}
\newtheorem{corollary}[ccorollary]{Corollary}
\newcounter{clemma}
\newtheorem{lemma}[clemma]{Lemma}

\newcounter{cremark}
\newtheorem{remark}[cremark]{Remark}

\newproof{proof}{Proof}

\journal{Journal of Multivariate Analysis}

\begin{document}

\begin{frontmatter}

\title{Robust spiked random matrices\\ and a robust G-MUSIC estimator\tnoteref{t1}}
\tnotetext[t1]{This work is jointly supported by the French ANR DIONISOS project (ANR-12-MONU-OOO3) and the GDR ISIS--GRETSI ``Jeunes Chercheurs'' Project.}

\author[supelec]{Romain Couillet}
\ead{romain.couillet@supelec.fr}
\address[supelec]{Telecommunication department, Sup\'elec, Gif sur Yvette, France}



\begin{abstract}
	A class of robust estimators of scatter applied to information-plus-impulsive noise samples is studied, where the sample information matrix is assumed of low rank; this generalizes the study \citep{COU13b} to spiked random matrix models. It is precisely shown that, as opposed to sample covariance matrices which may have asymptotically unbounded (eigen-)spectrum due to the sample impulsiveness, the robust estimator of scatter has bounded spectrum and may contain isolated eigenvalues which we fully characterize. We show that, if found beyond a certain detectability threshold, these eigenvalues allow one to perform statistical inference on the eigenvalues and eigenvectors of the information matrix. We use this result to derive new eigenvalue and eigenvector estimation procedures, which we apply in practice to the popular array processing problem of angle of arrival estimation. This gives birth to an improved algorithm based on the MUSIC method, which we refer to as robust G-MUSIC. 
\end{abstract}


\end{frontmatter}

\section{Introduction}
\label{sec:intro}

The mathematical advances in the field of random matrix theory have recently allowed for the improvement of sometimes old statistical estimation methods when the data have population size $N$ is commensurable with the sample size $n$, therefore disrupting the traditional assumption $n\gg N$. One of the recent contributions of random matrix theory lies in the introduction of methods to retrieve information contained in low rank perturbations of large matrices with independent entries, which are referred to as spiked models. The initial study of such models \citep{BAI06} for matrices of the type $\hat{S}_N=\frac1n(I_N+A)XX^*(I_N+A^*)$, where $X\in\CC^{N\times n}$ has independent and identically distributed (i.i.d.\@) zero mean, unit variance, and finite fourth moment entries and $A$ has fixed rank $L$, has shown that, as $N,n\to\infty$ with $N/n\to c\in (0,\infty)$, $\hat{S}_N$ may exhibit up to $L$ isolated eigenvalues strictly away from the {\it bounded} support of the limiting empirical distribution $\mu$ of $\hat{S}_N^\circ=\frac1nXX^*$, while the other eigenvalues of $\hat{S}_N$ get densely compacted in the support of $\mu$. This result has triggered multiple works on various low rank perturbation models for Gram, Wigner, or general square random matrices \citep{BEN09,PAU07,BEN10} with similar conclusions. Of particular interest to us here is the information-plus-noise model $\hat{S}_N=\frac1n(X+A)(X+A)^*$ introduced in \citep{BEN09} which is closer to our present model. Other generalizations explored the direction of turning $X$ into the more general $XT^{\frac12}$ model for $T=\diag(\tau_1,\ldots,\tau_n)\succeq 0$, such that $\frac1n\sum_{i=1}^n{\bm\delta}_{\tau_i}\to \nu$ weakly, where $\nu$ has bounded support ${\rm Supp}(\nu)$ and $\max_i\{{\rm dist}(\tau_i,{\rm Supp}(\nu))\}\to 0$ \citep{CHA12}. In this scenario again, thanks to the fundamental assumption that no $\tau_i$ can escape ${\rm Supp}(\nu)$ asymptotically, only finitely many eigenvalues of $\hat{S}_N$ can be found away from the support of the limiting spectral distribution of $\frac1nXTX^*$, and these eigenvalues are intimately linked to $A$. 

The major interest of the spiked models in practice is twofold. First, if the (non observable) perturbation matrix $A$ constitutes the relevant information to the system observer, then the observable isolated eigenvalues and associated eigenvectors of $\hat{S}_N$ contain information about $A$. These isolated eigenvalues and eigenvectors are therefore important objects to characterize. Moreover, since $\hat{S}_N$ has the same limiting spectrum as that of simple random matrix models, this characterization is usually quite easy and leads to tractable expressions and computationally efficient algorithms. This led to notable contributions to statistical inference and in particular to detection and estimation techniques for signal processing \citep{MES08b,NAD10,HAC13b,COU11e}. 

However, from the discussion of the first paragraph, these works have a few severe practical limitations in that: (i) the support of the limiting spectral distribution of $\hat{S}_N$ must be bounded for isolated eigenvalues to be detectable and exploitable and (ii) no eigenvalue of $\hat{S}_N^\circ$ (the unperturbed model) can be isolated, to avoid risking a confusion between isolated eigenvalues of $\hat{S}_N$ arising from $A$ and isolated eigenvalues of $\hat{S}_N$ intrinsically linked to $\hat{S}_N^\circ$. This therefore rules out the possibility to straightforwardly extend these techniques in practice to impulsive noise models $XT^{\frac12}$ where $T=\diag(\tau_1,\ldots,\tau_n)$ with either $\tau_i$ i.i.d.\@ arising from a distribution with unbounded support or $\tau_i=1$ for all but a few indices $i$. In the former case, the support of the limiting spectrum of $\hat{S}_N^\circ$ is unbounded \cite[Proposition~3.4]{HAC13}, therefore precluding information detection, while in the latter spurious eigenvalues in the spectrum of $\hat{S}_N$ may arise that are also found in $\hat{S}_N^\circ$ and therefore constitute false information (note that this case can be seen as one where low rank perturbations are present {\it both in the population and in the sample directions} which cannot be discriminated). Such impulsive models are nonetheless fundamental in many applications such as statistical finance or radar array processing, where impulsive samples are classically met.

Traditional statistical techniques to accommodate for impulsive samples fall in the realm of robust estimation \citep{MAR06book}, the study of which has long remained limited to the assumption $n\gg N$. Recently though, in a series of articles \citep{COU13,COU13b,COU14}, the author of the present article and his coauthors provided a random matrix analysis of robust estimation, i.e., assuming $N,n\to \infty$ and $N/n\to c\in(0,1)$, which revealed that robust sample estimates $\hat{C}_N^\circ$ of scatter (or covariance) matrices can be fairly easily analyzed through simpler equivalent random matrix models. In \citep{COU13b}, a noise-only setting of the present article is considered, i.e., with $A=0$, for which it is precisely shown that robust estimators of scatter can be assimilated as special models of the type of $\hat{S}_N^\circ$.\footnote{These models are special in that $XTX^*$ becomes now $XTVX^*$ for a diagonal matrix $V$ which makes $VT$ bounded in norm. However, $V$ contains non-observable information about $T$, which makes $\hat{S}_N^\circ$ only observable through its approximation by $\hat{C}_N^\circ$.} Besides, it importantly appears that the limiting spectrum distribution of $\hat{C}_N^\circ$ always has bounded support, irrespective of the impulsiveness of the samples. Also, it is proved (although not mentioned explicitly) that, asymptotically, isolated eigenvalues of $\hat{C}_N^\circ$ (arising from isolated $\tau_i$) can be found but that none of the eigenvalues can exceed a fixed finite value. 

In the present work, we extend the model studied in \citep{COU13b} by introducing a finite rank perturbation $A$ to the robust estimator of scale $\hat{C}_N^\circ$, the resulting matrix being denoted $\hat{C}_N$. As opposed to non-robust models, it shall appear (quite surprisingly on the onset) that $\hat{C}_N$ now allows for finitely many isolated eigenvalues to appear beyond the aforementioned fixed finite value (referred from now on to as the detection threshold), these eigenvalues being related to $A$. This holds even if $\frac1n\sum_{i=1}^n{\bm\delta}_{\tau_i}$ has unbounded support in the large $n$ regime. As such, any isolated eigenvalue of $\hat{C}_N$ found below the detection threshold may carry information about $A$ or may merely be an outlier due to an isolated $\tau_i$ (as in the non-robust context) but any eigenvalue found beyond the detection threshold necessarily carries information about $A$. This has important consequences in practice as now low rank perturbations {\it in the sample direction} are appropriately harnessed by the robust estimator while the (more relevant) low rank perturbations {\it in the population direction} can be properly estimated. We shall introduce an application of these results to array processing by providing two novel estimators for the power and steering direction of signals sources captured by a large sensor array under impulsive noise.

Our contribution thus lies on both theoretical and practical grounds. We first introduce in Theorem~\ref{th:1} the generalization of \citep{COU13b} to the perturbed model $\hat{C}_N$ which we precisely define in Section~\ref{sec:model}. The main results are then contained in Section~\ref{sec:results}. In this section, Theorem~\ref{th:2} provides the localization of the eigenvalues of $\hat{C}_N$ in the large system regime along with associated population eigenvalue and eigenvector estimators when the limiting distribution for $\frac1n\sum_{i=1}^n{\bm\delta}_{\tau_i}$ is known. This result is then extended in Theorem~\ref{th:3} thanks to a two-step estimator where the $\tau_i$ are directly estimated. A practical application of these novel methods to the context of steering angle estimation for array processing is then provided, leading to an improved algorithm referred to as robust G-MUSIC. Simulation results in this context are then displayed that confirm the improved performance of using robust schemes versus traditional sample covariance matrix-based techniques. We finally close the article with concluding remarks in Section~\ref{sec:conclusion}.

{\it Notations}: Vectors and matrices are represented in lower- and upper-case characters, respectively. Transpose and Hermitian transpose of $X$ are denoted respectively by $X^T$ and $X^*$. The norm $\Vert \cdot \Vert$ is the spectral norm for matrices and the Euclidean norm for vectors. The matrix $T^{\frac12}$ is the nonnegative definite square root of the Hermitian nonnegative definite matrix $T$. The eigenvalues of a Hermitian matrix $X\in\CC^{N\times N}$ are denoted in order as $\lambda_1(X)\geq \ldots \geq \lambda_N(X)$. Hermitian matrix ordering is denoted $X \succeq Y$, i.e., $X-Y$ is nonnegative definite. The support of a measure $\mu$ is denoted ${\rm Supp}(\mu)$. Almost sure convergence will be sometimes denoted ``$\asto$''. The Dirac measure at $x$ is denoted ${\bm\delta}_x$.

\section{Model and Motivation}
\label{sec:model}

Let $n\in \NN$. For $i\in\{1,\ldots,n\}$, we consider the following statistical model
\begin{align}
	\label{eq:model}
	y_i &= \sum_{l=1}^L \sqrt{p_l} a_l s_{li} + \sqrt{\tau_i} w_i
\end{align}
with $y_i\in \CC^N$ satisfying the following hypotheses.

\begin{assumption}
	\label{ass:y}
	The vectors $y_1,\ldots,y_n\in \CC^N$ satisfy the following conditions:
\begin{enumerate}
	\item \label{item:tau} $\tau_1,\ldots,\tau_n\in(0,\infty)$ are random scalars such that $\nu_n\triangleq\frac1n\sum_{i=1}^n {\bm\delta}_{\tau_i}\to \nu$ weakly, almost surely, where $\int t \nu(dt)=1$;
	\item $w_1,\ldots,w_n\in\CC^N$ are random independent unitarily invariant $\sqrt{N}$-norm vectors, independent of $\tau_1,\ldots,\tau_n$;
	\item $L\in \NN$, $p_1\geq \ldots\geq p_L\geq 0$ are deterministic and independent of $N$
	\item \label{item:a} $a_1,\ldots,a_L\in\CC^N$ are deterministic or random and such that
		\begin{align*}
			A^*A \asto \diag(p_1,\ldots,p_L)
		\end{align*}
		as $N\to\infty$, with $A\triangleq [\sqrt{p_1}a_1,\ldots,\sqrt{p_L} a_L]\in\CC^{N\times L}$
	\item $s_{1,1},\ldots,s_{Ln}\in \CC$ are independent with zero mean, unit variance, and uniformly bounded moments of all orders.
\end{enumerate}
\end{assumption}

For further use, we shall define
\begin{align*}
	A_i &\triangleq \begin{bmatrix} \sqrt{p_1}a_1&\ldots&\sqrt{p_L}a_L & \sqrt{\tau_i}I_N \end{bmatrix} \in\CC^{N\times (N+L)}.
\end{align*}
In particular, $A_iA_i^*=AA^*+\tau_i I_N$.

\begin{remark}[Application contexts]
	\label{rem:examples}
The system \eqref{eq:model} can be adapted to multiple scenarios in which the $s_{li}$ model scalar signals or data originated from $L$ sources of respective powers $p_1,\ldots,p_L$ carried by the vectors $a_1,\ldots,a_L$, while the $\sqrt{\tau_i}w_i$ model additive impulsive noise. Two examples are:
\begin{itemize}
	\item wireless communication channels in which signals $s_{li}$ originating from $L$ transmitters are captured by an $N$-antenna receiver. The vectors $a_l$ are here random independent channels for which it is natural to assume that $a_l^*a_{l'}\to {\bm\delta}_{l-l'}$ (e.g., for independent $a_l\sim \mathcal{CN}(0,I_N/N)$);
	\item array processing in which $L$ sources emit signals $s_{li}$ captured by an antenna array through steering vectors $a_l=a(\theta_l)$ for a given $a(\theta)$ function and angles of arrival $\theta_1,\ldots,\theta_L\in[0,2\pi)$. In the case of uniform linear arrays with inter-antenna distance $d$, $[a(\theta)]_j=N^{-\frac12}\exp(2\pi\imath dj\sin(\theta))$.
\end{itemize}
The noise impulsiveness is translated by the $\tau_i$ coefficients. The vectors $\sqrt{\tau_i}w_i$ are for instance i.i.d.\@ elliptic random vectors if the $\tau_i$ are i.i.d.\@ with absolutely continuous measure $\tilde\nu_n$ having a limit $\tilde{\nu}$ (in which case, we easily verify that $\nu_n\to \nu=\tilde\nu$ almost surely (a.s.)). This particularizes to additive white Gaussian noise if $2N\tau_i$ is chi-square with $2N$ degrees of freedom (in this case, $\nu={\bm\delta}_1$). Of interest in this article is however the scenarios where $\nu$ has unbounded support, e.g., when the $\tau_i$ are either random i.i.d.\@ and heavy-tailed or contain a few arbitrarily large outliers, which both correspond to impulsive noise scenarios.
\end{remark}

\begin{remark}[Technical comments]
	From a purely technical perspective, it is easily seen from the proofs of our main results in Section~\ref{sec:proof} that some of the items of Assumption~\ref{ass:y} could have been relaxed. In particular, Item~(\ref{item:a}) could have been relaxed into ``all accumulation points of $A^*A$ are similar to $\diag(q_1,\ldots,q_L)$ for given $q_1\geq \ldots\geq q_L$'' as in e.g., \citep{CHA12}. Also, similar to \citep{COU13b}, the convergence of $\nu_n$ in Item~(\ref{item:tau}) could be relaxed to the cost of introducing a tightness condition on the sequence $\{\nu_n\}_{n=1}^\infty$ and to loose the convergence of measure in the discussion following Theorem~\ref{th:1}. For readability and since Assumption~\ref{ass:y} gathers most of the scenarios of interest, we restrict ourselves to those (already quite general) hypotheses.
\end{remark}

We now define the robust estimate of scatter $\hat{C}_N$. We start by denoting $u:[0,\infty)\to (0,\infty)$ any function satisfying the following hypotheses.
\begin{assumption}
	\label{ass:u}
	The function $u$ is characterized by
	\begin{enumerate}
		\item $u$ is continuous, nonnegative, and non-increasing from $[0,\infty)$ onto $(0,u(0)]\subset (0,\infty)$;
		\item for $x\geq 0$, $\phi(x)\triangleq xu(x)$ is increasing and bounded with
			\begin{align*}
				\phi_\infty &\triangleq \lim_{x\to\infty}\phi(x)>1
			\end{align*}
		\item there exists $m>0$ such that $\nu([0,m))<1-\phi_\infty^{-1}$;
		\item for all $a>b>0$,
			\begin{align*}
				\limsup_{t\to\infty} \frac{\nu( (t,\infty) )}{\phi(at)-\phi(bt)} = 0.
			\end{align*}
	\end{enumerate}
\end{assumption}
These assumptions are the same as in \cite{COU13b} which are therefore not altered by the updated model \eqref{eq:model}.

The function $u$ being given, we now define $\hat{C}_N$, when it exists, as the unique solution to the fixed-point matrix-valued equation in $Z$:
\begin{align*}
	Z = \frac1n\sum_{i=1}^n u\left( \frac1N y_i^* Z^{-1} y_i \right) y_iy_i^*.
\end{align*}
For $i\in\{1,\ldots,N\}$, we shall denote $\hat{\lambda}_i\triangleq \lambda_i(\hat{C}_N)$ and $\hat{u}_i\in\CC^N$ the $i$-th largest eigenvalue of $\hat{C}_N$ and its associated eigenvector. 

Due to its implicit formulation, the study of $\hat{C}_N$ for every fixed $N,n$ couple is quite involved in general. As such, similar to \citep{COU13b}, we shall place ourselves in the regime where both $N$ and $n$ are large but with non trivial ratio. Hence, we shall assume the following system growth regime.

\begin{assumption}
	\label{ass:c}
	The integer $N=N(n)$ is such that $c_n\triangleq N/n$ satisfies 
	\begin{align*}
		\lim_{n\to\infty} c_n = c\in (0,\phi_\infty^{-1}).
	\end{align*}
	Meanwhile, $L$ remains constant independently of $N,n$.
\end{assumption}

Up to differences in the hypotheses of Assumption~\ref{ass:u} and Assumption~\ref{ass:c}, and a slight difference in notations, $\hat{C}_N$ is exactly the robust estimator of scatter proposed by Maronna in \citep{MAR76}.
As a direct application of \citep{CHI14}, under Assumption~\ref{ass:y} and Assumption~\ref{ass:u}, $\hat{C}_N$ is almost surely well defined for each couple $N,n$ with $N<n$. Also, from \citep{COU13b}, $\hat{C}_N$ can be written (at least for all large $n$) in the technically more convenient form (see discussions in \citep{COU13b})
\begin{align*}
	\hat{C}_N &= \frac1n\sum_{i=1}^n v\left( \frac1n y_i^* \hat{C}_{(i)}^{-1} y_i \right) y_iy_i^*
\end{align*}
where $v:x\mapsto u\circ g^{-1}$, $g:x\mapsto x/(1-c_n\phi(x))$, and $\hat{C}_{(i)}=\hat{C}_N-u\left( \frac1n y_i^* \hat{C}_N^{-1} y_i \right) y_iy_i^*$. We shall further denote $\psi(x)=xv(x)$. It is easy to see that $v$ is non-increasing while $\psi$ is increasing with limit $\psi_\infty=\phi_\infty/(1-c_n\phi_\infty)$.

With these definitions in place, we are now in position to present our main results.

\section{Main Results}
\label{sec:results}

The first objective of the article is to study the spectrum of $\hat{C}_N$ and in particular its largest eigenvalues $\hat{\lambda}_1\geq \ldots\geq \hat{\lambda}_L$ and associated eigenvectors $\hat{u}_1,\ldots,\hat{u}_L$, in the large $N,n$ regime. This study will in turn allow us to retrieve information on $p_1,\ldots,p_L$ and $a_1,\ldots,a_L$. As an application, a novel improved angle estimator for array processing will then be provided.

\subsection{Localisation and estimation}

Our first result is an extension of \cite[Theorem~2]{COU13b} which states that $\hat{C}_N$, the implicit structure of which makes it complicated to analyze, can be appropriately replaced by a more practical random matrix $\hat{S}_N$, which is much easier to study.

\begin{theorem}[Asymptotic model equivalence]
	\label{th:1}
	Let Assumptions~\ref{ass:y},~\ref{ass:u},~and~\ref{ass:c} hold. Then
	\begin{align*}
		\Vert \hat{C}_N - \hat{S}_N \Vert \asto 0
	\end{align*}
	where
	\begin{align*}
		\hat{S}_N &\triangleq \frac1n \sum_{i=1}^n v_c(\tau_i \gamma) A_i \bar{w}_i \bar{w}_i^* A_i^*
	\end{align*}
	with $\gamma$ the unique solution to
	\begin{align*}
		1 &= \int \frac{\psi_c(t\gamma)}{1+c\psi_c(t\gamma)}\nu(dt)
	\end{align*}
	$v_c$ and $\psi_c$ the limits of $v$ and $\psi$ as $c_n\to c$, and $\bar{w}_i=[s_{1i},\ldots,s_{Li},w_i r_i/\sqrt{N}]^\trans$, with $r_i\geq 0$ such that $2Nr_i^2$ is a chi-square random variable with $2N$ degrees of freedom, independent of $w_i$.%
	\footnote{Note that $w_i r_i/\sqrt{N}$ as defined above is a standard Gaussian vector and therefore $\bar{w}_i$ has independent entries of zero mean and unit variance. In fact, the result can be equivalently formulated with $\bar{w}_i$ replaced by $[s_{1i},\ldots,s_{Li},w_i]^\trans$, but the former vector, having independent entries, is of more interest statistically.}
\end{theorem}

\begin{remark}[From robust estimator to sample covariance matrix]
	\label{rem:v=1}
	Note that, if the function $v_c$ in the expression of $\hat{S}_N$ were replaced by the constant $1$ (and $r_i/\sqrt{N}$ set to one), $\hat{S}_N$ would be the classical sample covariance matrix of $y_1,\ldots,y_n$. Although it is here highly non rigorous to let $v_c$ tend to $1$ uniformly in Theorem~\ref{th:1}, this remark somewhat reveals the classical robust estimation intuition according to which the larger $\phi_\infty$ (as a consequence of $u$ and $v_c$ being close to $1$) the less robust $\hat{C}_N$.
\end{remark}

As a corollary of Theorem~\ref{th:1}, we have
\begin{align}
	\label{eq:cvg_lambdas}
	\max_{1\leq i\leq N}\left|\hat{\lambda}_i-\lambda_i(\hat{S}_N)\right|\asto 0
\end{align}
(which unfolds from applying \cite[Theorem~4.3.7]{HOR85}) and therefore all eigenvalues of $\hat{C}_N$ can be accurately controlled through the eigenvalues of $\hat{S}_N$. 

\bigskip

Let us assume for a moment that $p_1=\ldots=p_L=0$. Then, from Theorem~\ref{th:1}, Assumption~\ref{ass:y}, and \citep{CHO95}, $\mu_n\triangleq \frac1N\sum_{i=1}^N {\bm\delta}_{\hat\lambda_i}\to \mu$ weakly, a.s., where $\mu$ has a density on $\RR$ with bounded support ${\rm Supp}(\mu)\subset \RR^+$. Denote
\begin{align*}
	S^-_\mu &\triangleq \inf({\rm Supp}(\mu)) \\
	S^+_\mu &\triangleq \sup({\rm Supp}(\mu)) \\
	S^+ &\triangleq \frac{\phi_\infty(1+\sqrt{c})^2}{\gamma(1-c\phi_\infty)}.
\end{align*}
Since $\tau_i v_c(\tau_i\gamma)=\gamma^{-1}\psi_c(\tau_i\gamma)<\gamma^{-1}\psi_{c,\infty}$ with $\psi_{c,\infty}=\phi_\infty/(1-c\phi_\infty)$, we have
\begin{align*}
	\hat{S}_N &\preceq \frac{\phi_\infty}{\gamma(1-c\phi_\infty)}\frac1n\sum_{i=1}^n w_iw_i^*
\end{align*}
so that, according to \citep{MAR67,SIL98} and \eqref{eq:cvg_lambdas}, for each $\varepsilon>0$, $\hat{\lambda}_1<S^++\varepsilon$ for all large $n$ a.s. Of course, $S^+\geq S_\mu^+$. If in addition $\max_{1\leq i\leq n}\{{\rm dist}(\tau_i,{\rm Supp}(\nu))\}\asto 0$, then from \citep{SIL98}, we even have $\hat{\lambda}_1\asto S^+_\mu$; but this constraint is of little practical interest so that in general one may have $S^+_\mu <\hat{\lambda}_1<S^+$ infinitely often.

\bigskip

Coming back to generic values for $p_1,\ldots,p_L$, the idea of the results below is that, for sufficiently large $p_1,\ldots,p_L$, the eigenvalues $\hat{\lambda}_1,\ldots,\hat{\lambda}_L$ may exceed $S^++\varepsilon$ and contain information to estimate $p_1,\ldots,p_L$ as well as bilinear forms involving $a_1,\ldots,a_L$. The exact location of the eigenvalues and the value of these estimates shall be expressed as a function of the fundamental object $\delta(x)$, defined for $x\in \RR^* \setminus [S_\mu^-,S_{\mu}^+]$ as the unique real solution to
\begin{align*}
	\delta(x) = c\left( - x + \int \frac{tv_c(t\gamma)}{1+\delta(x)tv_c(t\gamma)}\nu(dt) \right)^{-1}.
\end{align*}
The function $\delta(x)$ is the restriction to $\RR^* \setminus [S_\mu^-,S_{\mu}^+]$ of the Stieltjes transform of $c\mu+(1-c){\bm\delta}_0$ and is, as such, increasing on $(S^+,\infty)\subset (S_\mu^+,\infty)$
; see \citep{CHO95,HAC13} and Section~\ref{sec:proof} for details. Therefore, the following definition of $p_-$, which will be referred to as the detectability threshold, is licit
\begin{align*}
	p_- \triangleq \lim_{x \downarrow S^+} -c \left( \int \frac{\delta(x)v_c(t\gamma)}{1+\delta(x)tv_c(t\gamma)}\nu(dt) \right)^{-1}.
\end{align*}
We shall further denote $\mathcal L\triangleq \{j,p_j>p_-\}$. 

\bigskip

We are now in position to provide our main results.
\begin{theorem}[Robust estimation under known $\nu$]
	\label{th:2}
	Let Assumptions~\ref{ass:y},~\ref{ass:u}, and \ref{ass:c} hold.
	Denote $u_k$ the eigenvector associated with the $k$-th largest eigenvalue of $AA^*$ (in case of multiplicity, take any vector in the eigenspace with $u_1,\ldots,u_L$ orthogonal) and $\hat{u}_1,\ldots,\hat{u}_N$ the eigenvectors of $\hat{C}_N$ respectively associated with the eigenvalues $\hat{\lambda}_1\geq \ldots \geq \hat{\lambda}_N$. Then, we have the following three results.

	\medskip
	
	{\sf 0. Extreme eigenvalues.} For each $j\in\mathcal L$,
	\begin{align*}
		\hat{\lambda}_j &\asto \Lambda_j > S^+
	\end{align*}
	while $\limsup_n \hat{\lambda}_{|\mathcal L|+1} \leq S^+$ a.s., where $\Lambda_j$ is the unique positive solution to
	\begin{align*}
		-c\left( \delta(\Lambda_j) \int \frac{v_c(\tau \gamma)}{1+\delta(\Lambda_j)\tau v_c(\tau\gamma)} \nu(d\tau)\right)^{-1}= p_j.
	\end{align*}
	
	{\sf 1. Power estimation.} For each $j\in\mathcal L$,
	\begin{align*}
		- c \left( \delta(\hat{\lambda}_j) \int \frac{v_c(\tau\gamma)}{1+\delta(\hat{\lambda}_j)\tau v_c(\tau \gamma)}\nu(d\tau) \right)^{-1}\asto p_j.
	\end{align*}

	{\sf 2. Bilinear form estimation.}
	For each $a,b\in\CC^N$ with $\Vert a\Vert=\Vert b\Vert=1$, and $j\in\mathcal L$
	\begin{align*}
	 	\sum_{k,p_k=p_j} a^*u_ku_k^* b - \sum_{k,p_k=p_j} w_k a^* \hat{u}_k\hat{u}_k^* b  \asto 0
	\end{align*}
	where
	\begin{align*}
		w_k &= \frac{ \displaystyle\int \frac{v_c(t\gamma)}{\left(1+\delta(\hat{\lambda}_k)tv_c(t\gamma)\right)^2}  \nu(dt)}{\displaystyle \int \frac{v_c(t\gamma)}{1+\delta(\hat{\lambda}_k)tv_c(t\gamma)}  \nu(dt)\left( 1 - \frac1c \int \frac{\delta(\hat{\lambda}_k)^2t^2v_c(t\gamma)^2}{\left(1+\delta(\hat{\lambda}_k)tv_c(t\gamma)\right)^2}  \nu(dt) \right)}.
	\end{align*}
\end{theorem}

Item 0.\@ in Theorem~\ref{th:2} provides a necessary and sufficient condition, i.e., $p_j>p_-$, for the existence of outlying eigenvalues in the spectrum of $\hat{C}_N$. In turn, this provides a means to estimate each $p_j$, $j\in\mathcal L$, along with bilinear forms involving $a_j$, from $\hat{\lambda}_j$ and $\hat{u}_j$. It is important here to note that, although the right-edge of the spectrum of $\mu$ is $S_\mu^+$, due to the little control on $\tau_i$ in practice (in particular some of the $\tau_i$ may freely be arbitrarily large), isolated eigenvalues may be found infinitely often beyond $S_\mu^+$ which do not carry information. This is why the (possibly pessimistic) choice of $S^+$ as an eigenvalue discrimination threshold was made. The major potency of the robust estimator $\hat{C}_N$ is indeed to be able to maintain these non informative eigenvalues below the known value $S^+$. As such, eigenvalues found above $S^+$ must contain information about $A$ (at least with high probability) and this information can be retrieved, while isolated eigenvalues found below $S^+$ may arise from spurious values of $\tau_i$, therefore containing no relevant information, or may contain relevant information but that cannot be trusted. 

Figure~\ref{fig:histo_C} and Figure~\ref{fig:histo_YY} provide the histogram and limiting spectral distribution of $\hat{C}_N$ and $\frac1nYY^*$, $Y=[y_1,\ldots,y_n]$, respectively, for $u(x)=(1+\alpha)/(\alpha+x)$, $\alpha=0.2$, $N=200$, $n=1000$, $\tau_i$ i.i.d.\@ equal in distribution to $t^2(\beta-2)\beta^{-1}$ with $t$ a Student-t random scalar of parameter $\beta=100$, and $L=2$ with $p_1=p_2=1$, $a_1=a(\theta_1)$, $a_2=a(\theta_2)$, $\theta_1=10^\circ$, $\theta_2=12^\circ$, $a(\theta)$ being defined in Remark~\ref{rem:examples} (as well as in Assumption~\ref{ass:arrayprocessing} below). These curves confirm that, while the limiting spectral measure of $\frac1nYY^*$ is unbounded, that of $\hat{C}_N$ is bounded. The numerically evaluated values of $S_\mu^+$ and $S^+$ are reported in Figure~\ref{fig:histo_C}. They reveal a rather close proximity between both values. In terms of empirical eigenvalues, note the particularly large gap between the isolated eigenvalues of $\hat{C}_N$ and the $N-2$ smallest ones, which may seem at first somewhat surprising for $p_1=p_2=1$ since this setting induces a ratio $1$ between the power carried by information versus noise (indeed, $A^*A\simeq I_2$ while $\EE[\tau_iw_iw_i^*]=I_N$); this in fact results from the function $u$ which, in attenuating the rare samples of large amplitudes, significantly reduces the noise power but only weakly affects the information part which has roughly constant amplitude across the samples. Also observe from Figure~\ref{fig:histo_YY} that, as predicted, the largest two eigenvalues of $\frac1nYY^*$ do not isolate from the majority of the eigenvalues.

\begin{figure}[h!]
  \centering
  \begin{tikzpicture}[font=\footnotesize]
    \renewcommand{\axisdefaulttryminticks}{4} 
    \tikzstyle{every major grid}+=[style=densely dashed]       
    \tikzstyle{every axis y label}+=[yshift=-10pt] 
    \tikzstyle{every axis x label}+=[yshift=5pt]
    \tikzstyle{every axis legend}+=[cells={anchor=west},fill=white,
        at={(0.98,0.98)}, anchor=north east, font=\scriptsize ]
    \begin{axis}[
      xmin=0,
      ymin=0,
      xmax=1.2,
      ymax=8,
      grid=major,
      ymajorgrids=false,
      scaled ticks=true,
      bar width=1pt,
      xlabel={Eigenvalues},
      ylabel={Density}
      ]
      \addplot[ybar,area legend,fill=lightgray] plot coordinates{
(0.035,1.000000)(0.045,5.000000)(0.055,7.500000)(0.065,6.500000)(0.075,6.000000)(0.085,6.000000)(0.095,6.500000)(0.105,6.500000)(0.115,5.000000)(0.125,5.000000)(0.135,5.500000)(0.145,4.000000)(0.155,4.500000)(0.165,4.000000)(0.175,3.500000)(0.185,4.000000)(0.195,3.000000)(0.205,2.500000)(0.215,3.500000)(0.225,2.500000)(0.235,2.000000)(0.245,2.000000)(0.255,1.000000)(0.265,1.500000)(0.275,0.500000)(0.975,0.500000)(1.155,0.500000)
      };
      \addplot[black,line width=1pt] plot coordinates{
(0.034,0.0)(0.035,0.439542)(0.036,2.510349)(0.037,3.440712)(0.038,4.161353)(0.039,4.624231)(0.040,5.031836)(0.041,5.380654)(0.042,5.651088)(0.043,5.915932)(0.044,6.087717)(0.045,6.272923)(0.046,6.434383)(0.047,6.552407)(0.048,6.659086)(0.049,6.735943)(0.050,6.827175)(0.051,6.879654)(0.052,6.947642)(0.053,6.993367)(0.054,7.041975)(0.055,7.076681)(0.056,7.081445)(0.057,7.121571)(0.058,7.131253)(0.059,7.145184)(0.060,7.133597)(0.061,7.134712)(0.062,7.149019)(0.063,7.131673)(0.064,7.126910)(0.065,7.127476)(0.066,7.108826)(0.067,7.098651)(0.068,7.083594)(0.069,7.051320)(0.070,7.040792)(0.071,7.023223)(0.072,6.987149)(0.073,6.964919)(0.074,6.953972)(0.075,6.919333)(0.076,6.888480)(0.077,6.867257)(0.078,6.831701)(0.079,6.812041)(0.080,6.772036)(0.081,6.756584)(0.082,6.725277)(0.083,6.691205)(0.084,6.655820)(0.085,6.631118)(0.086,6.585424)(0.087,6.564353)(0.088,6.528563)(0.089,6.491112)(0.090,6.464228)(0.091,6.433983)(0.092,6.391470)(0.093,6.353344)(0.094,6.322149)(0.095,6.295199)(0.096,6.259295)(0.097,6.230740)(0.098,6.187958)(0.099,6.151201)(0.100,6.115372)(0.101,6.093490)(0.102,6.047301)(0.103,6.019986)(0.104,5.991129)(0.105,5.956934)(0.106,5.912574)(0.107,5.884253)(0.108,5.842347)(0.109,5.819346)(0.110,5.783926)(0.111,5.751665)(0.112,5.706455)(0.113,5.680853)(0.114,5.650956)(0.115,5.609900)(0.116,5.581258)(0.117,5.539023)(0.118,5.509720)(0.119,5.483611)(0.120,5.440488)(0.121,5.416076)(0.122,5.382333)(0.123,5.340904)(0.124,5.316446)(0.125,5.279370)(0.126,5.241931)(0.127,5.219632)(0.128,5.175879)(0.129,5.144547)(0.130,5.113915)(0.131,5.086139)(0.132,5.046811)(0.133,5.021367)(0.134,4.983882)(0.135,4.951360)(0.136,4.929986)(0.137,4.887541)(0.138,4.858840)(0.139,4.835850)(0.140,4.804666)(0.141,4.768260)(0.142,4.732630)(0.143,4.702090)(0.144,4.679983)(0.145,4.638717)(0.146,4.609061)(0.147,4.587221)(0.148,4.546211)(0.149,4.525794)(0.150,4.487385)(0.151,4.457159)(0.152,4.433372)(0.153,4.396939)(0.154,4.373181)(0.155,4.332894)(0.156,4.309643)(0.157,4.283036)(0.158,4.244244)(0.159,4.222754)(0.160,4.183256)(0.161,4.160121)(0.162,4.122784)(0.163,4.095409)(0.164,4.063474)(0.165,4.045347)(0.166,4.008367)(0.167,3.975641)(0.168,3.946131)(0.169,3.922266)(0.170,3.894003)(0.171,3.868395)(0.172,3.830833)(0.173,3.804565)(0.174,3.779781)(0.175,3.752085)(0.176,3.717402)(0.177,3.687767)(0.178,3.663681)(0.179,3.625893)(0.180,3.601504)(0.181,3.569784)(0.182,3.548727)(0.183,3.515673)(0.184,3.481307)(0.185,3.452636)(0.186,3.434195)(0.187,3.400548)(0.188,3.370254)(0.189,3.335470)(0.190,3.306648)(0.191,3.289306)(0.192,3.251775)(0.193,3.226414)(0.194,3.194332)(0.195,3.172839)(0.196,3.145445)(0.197,3.105569)(0.198,3.087284)(0.199,3.049657)(0.200,3.028130)(0.201,2.991310)(0.202,2.966255)(0.203,2.933301)(0.204,2.901840)(0.205,2.873169)(0.206,2.850956)(0.207,2.824951)(0.208,2.795224)(0.209,2.755528)(0.210,2.730848)(0.211,2.705981)(0.212,2.666418)(0.213,2.638540)(0.214,2.606931)(0.215,2.579346)(0.216,2.556726)(0.217,2.523739)(0.218,2.483348)(0.219,2.463473)(0.220,2.422561)(0.221,2.401499)(0.222,2.359245)(0.223,2.339083)(0.224,2.298503)(0.225,2.267202)(0.226,2.243568)(0.227,2.212097)(0.228,2.170318)(0.229,2.135987)(0.230,2.106770)(0.231,2.084102)(0.232,2.038273)(0.233,2.017561)(0.234,1.968173)(0.235,1.944227)(0.236,1.899766)(0.237,1.864208)(0.238,1.843946)(0.239,1.800254)(0.240,1.758676)(0.241,1.724252)(0.242,1.693442)(0.243,1.644469)(0.244,1.623037)(0.245,1.578750)(0.246,1.528462)(0.247,1.486242)(0.248,1.444239)(0.249,1.401819)(0.250,1.364784)(0.251,1.314473)(0.252,1.267924)(0.253,1.222053)(0.254,1.172143)(0.255,1.121536)(0.256,1.069021)(0.257,1.014139)(0.258,0.965222)(0.259,0.896360)(0.260,0.831361)(0.261,0.761482)(0.262,0.688268)(0.263,0.600249)(0.264,0.503538)(0.265,0.385427)(0.266,0.230654)(0.267,0.072993)(0.268,0.024278)(0.269,0.012774)(0.270,0.008801)(0.271,0.007346)(0.272,0.006380)(0.273,0.005790)(0.274,0.005268)(0.275,0.004840)(0.276,0.004491)(0.277,0.004200)(0.278,0.003951)(0.279,0.003736)
      };
      \node at (axis cs:0.45,2.5) {\tiny $S^+$};
      \draw[<-] (axis cs:0.3191,0) |- (axis cs:0.45,2.2);
      \node at (axis cs:0.45,3.5) {\tiny $S_\mu^+$};
      \draw[<-] (axis cs:0.27,0) |- (axis cs:0.45,3.2);
      \node at (axis cs:1.1,3.5) {\tiny $\Lambda_1$};
      \draw[<-] (axis cs:1.1,0) -- (axis cs:1.1,3.2);
      \legend{ {Eigenvalues of $\hat{C}_N$},{Limiting spectral measure $\mu$} }
    \end{axis}
  \end{tikzpicture}
  \caption{Histogram of the eigenvalues of $\hat{C}_N$ against the limiting spectral measure, for $u(x)=(1+\alpha)/(\alpha+x)$ with $\alpha=0.2$, $L=2$, $p_1=p_2=1$, $N=200$, $n=1000$, Student-t impulsions.}
  \label{fig:histo_C}
\end{figure}
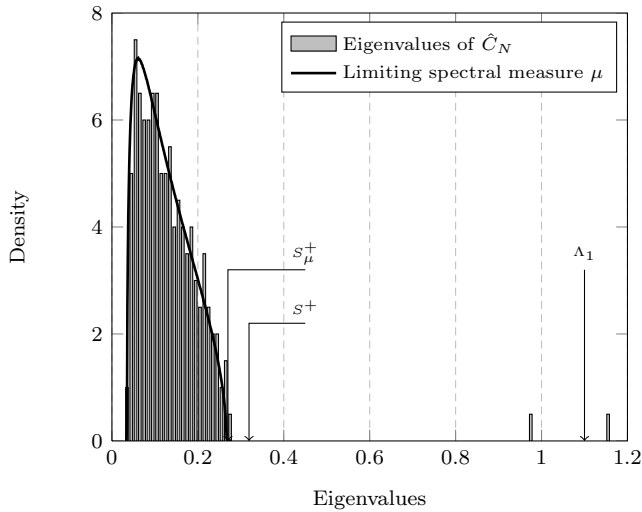

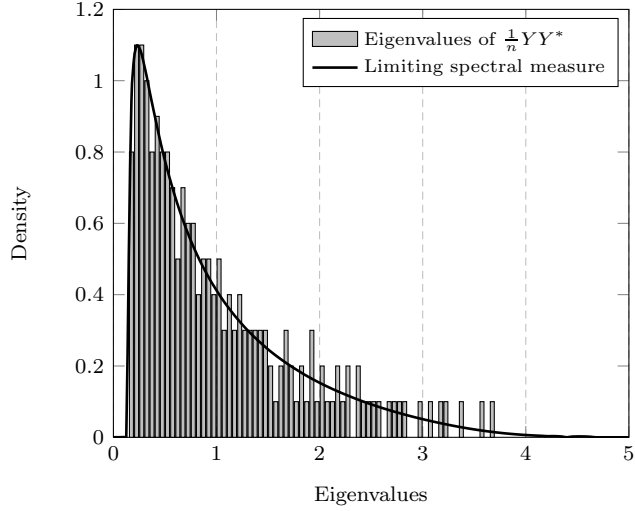
\begin{figure}[h!]
  \centering
  \begin{tikzpicture}[font=\footnotesize]
    \renewcommand{\axisdefaulttryminticks}{4} 
    \tikzstyle{every major grid}+=[style=densely dashed]       
    \tikzstyle{every axis y label}+=[yshift=-10pt] 
    \tikzstyle{every axis x label}+=[yshift=5pt]
    \tikzstyle{every axis legend}+=[cells={anchor=west},fill=white,
        at={(0.98,0.98)}, anchor=north east, font=\scriptsize ]
    \begin{axis}[
      xmin=0,
      ymin=0,
      xmax=5,
      ymax=1.2,
      grid=major,
      ymajorgrids=false,
      scaled ticks=true,
      bar width=1.5pt,
      xlabel={Eigenvalues},
      ylabel={Density}
      ]
      \addplot[ybar,area legend,fill=lightgray] plot coordinates{
	      (0.175,0.800000)(0.225,1.100000)(0.275,1.100000)(0.325,1.000000)(0.375,0.800000)(0.425,0.900000)(0.475,0.800000)(0.525,0.800000)(0.575,0.700000)(0.625,0.500000)(0.675,0.700000)(0.725,0.600000)(0.775,0.600000)(0.825,0.400000)(0.875,0.500000)(0.925,0.500000)(0.975,0.400000)(1.025,0.500000)(1.075,0.300000)(1.125,0.400000)(1.175,0.300000)(1.225,0.400000)(1.275,0.300000)(1.325,0.300000)(1.375,0.300000)(1.425,0.300000)(1.475,0.300000)(1.525,0.200000)(1.575,0.100000)(1.625,0.200000)(1.675,0.300000)(1.725,0.200000)(1.775,0.100000)(1.825,0.200000)(1.875,0.100000)(1.925,0.300000)(1.975,0.100000)(2.025,0.200000)(2.075,0.100000)(2.125,0.100000)(2.175,0.200000)(2.225,0.100000)(2.275,0.200000)(2.375,0.200000)(2.425,0.100000)(2.475,0.100000)(2.525,0.100000)(2.575,0.100000)(2.675,0.100000)(2.725,0.100000)(2.775,0.100000)(2.825,0.100000)(2.975,0.100000)(3.075,0.100000)(3.175,0.100000)(3.225,0.100000)(3.375,0.100000)(3.575,0.100000)(3.675,0.100000)
      };
      \addplot[black,smooth,line width=1pt] plot coordinates{
(0.000,0.000171)(0.025,0.000212)(0.050,0.000273)(0.075,0.000374)(0.100,0.000579)(0.125,0.001207)(0.150,0.517830)(0.175,0.942715)(0.200,1.061325)(0.225,1.097338)(0.250,1.093024)(0.275,1.073534)(0.300,1.041905)(0.325,1.009186)(0.350,0.972110)(0.375,0.937074)(0.400,0.903203)(0.425,0.869455)(0.450,0.836516)(0.475,0.805259)(0.500,0.776209)(0.525,0.748698)(0.550,0.722593)(0.575,0.697634)(0.600,0.673586)(0.625,0.650996)(0.650,0.629083)(0.675,0.608769)(0.700,0.588760)(0.725,0.570339)(0.750,0.552241)(0.775,0.535375)(0.800,0.519289)(0.825,0.503724)(0.850,0.488866)(0.875,0.474665)(0.900,0.461346)(0.925,0.448148)(0.950,0.435510)(0.975,0.423876)(1.000,0.412195)(1.025,0.400933)(1.050,0.390083)(1.075,0.379629)(1.100,0.369788)(1.125,0.360155)(1.150,0.350815)(1.175,0.341772)(1.200,0.333015)(1.225,0.324534)(1.250,0.316320)(1.275,0.308275)(1.300,0.300645)(1.325,0.293254)(1.350,0.286077)(1.375,0.279102)(1.400,0.272320)(1.425,0.265724)(1.450,0.259307)(1.475,0.253062)(1.500,0.246660)(1.525,0.240784)(1.550,0.235074)(1.575,0.229517)(1.600,0.224105)(1.625,0.218832)(1.650,0.213693)(1.675,0.208680)(1.700,0.203789)(1.725,0.199015)(1.750,0.194352)(1.775,0.189797)(1.800,0.185345)(1.825,0.180994)(1.850,0.176469)(1.875,0.172328)(1.900,0.168265)(1.925,0.164295)(1.950,0.160416)(1.975,0.156623)(2.000,0.152914)(2.025,0.149285)(2.050,0.145735)(2.075,0.142261)(2.100,0.138859)(2.125,0.135528)(2.150,0.132264)(2.175,0.129067)(2.200,0.125933)(2.225,0.122860)(2.250,0.119847)(2.275,0.116891)(2.300,0.113991)(2.325,0.111083)(2.350,0.108260)(2.375,0.105503)(2.400,0.102798)(2.425,0.100144)(2.450,0.097540)(2.475,0.094986)(2.500,0.092479)(2.525,0.090019)(2.550,0.087605)(2.575,0.085234)(2.600,0.082908)(2.625,0.080623)(2.650,0.078380)(2.675,0.076178)(2.700,0.074015)(2.725,0.071891)(2.750,0.069804)(2.775,0.067755)(2.800,0.065742)(2.825,0.063765)(2.850,0.061822)(2.875,0.059914)(2.900,0.058039)(2.925,0.056198)(2.950,0.054389)(2.975,0.052611)(3.000,0.050970)(3.025,0.049155)(3.050,0.047542)(3.075,0.045871)(3.100,0.044230)(3.125,0.042620)(3.150,0.041041)(3.175,0.039492)(3.200,0.037973)(3.225,0.036486)(3.250,0.035028)(3.275,0.033602)(3.300,0.032206)(3.325,0.030841)(3.350,0.029508)(3.375,0.028205)(3.400,0.026934)(3.425,0.025695)(3.450,0.024488)(3.475,0.023313)(3.500,0.022171)(3.525,0.021061)(3.550,0.019985)(3.575,0.018942)(3.600,0.017933)(3.625,0.016958)(3.650,0.016017)(3.675,0.015111)(3.700,0.014239)(3.725,0.013402)(3.750,0.012600)(3.775,0.011833)(3.800,0.011100)(3.825,0.010402)(3.850,0.009771)(3.875,0.009136)(3.900,0.008540)(3.925,0.007975)(3.950,0.007442)(3.975,0.006938)(4.000,0.006466)(4.025,0.006023)(4.050,0.005600)(4.075,0.005197)(4.100,0.004849)(4.125,0.004557)(4.150,0.004236)(4.175,0.003831)(4.200,0.003314)(4.225,0.003206)(4.250,0.003344)(4.275,0.003304)(4.300,0.003111)(4.325,0.002740)(4.350,0.002216)(4.375,0.000944)(4.400,0.000133)(4.425,0.001138)(4.450,0.002045)(4.475,0.002311)(4.500,0.002483)(4.525,0.002563)(4.550,0.002427)(4.575,0.002341)(4.600,0.002017)(4.625,0.001584)(4.650,0.000827)(4.675,0.000127)(4.700,0.000012)(4.725,0.000005)(4.750,0.000004)(4.775,0.000003)(4.800,0.000003)(4.825,0.000003)(4.850,0.000003)(4.875,0.000003)(4.900,0.000003)(4.925,0.000003)(4.950,0.000003)(4.975,0.000003)(5.000,0.000003)(5.025,0.000003)(5.050,0.000002)(5.075,0.000002)(5.100,0.000002)(5.125,0.000002)(5.150,0.000002)(5.175,0.000002)(5.200,0.000002)(5.225,0.000002)(5.250,0.000002)(5.275,0.000002)(5.300,0.000002)(5.325,0.000002)(5.350,0.000002)(5.375,0.000002)(5.400,0.000002)(5.425,0.000002)(5.450,0.000002)(5.475,0.000002)(5.500,0.000002)(5.525,0.000002)(5.550,0.000002)(5.575,0.000002)(5.600,0.000002)(5.625,0.000002)(5.650,0.000002)(5.675,0.000002)(5.700,0.000002)(5.725,0.000002)(5.750,0.000002)(5.775,0.000002)(5.800,0.000002)(5.825,0.000002)(5.850,0.000002)(5.875,0.000002)(5.900,0.000001)(5.925,0.000001)(5.950,0.000001)(5.975,0.000001)(6.000,0.000001)(6.025,0.000001)(6.050,0.000001)(6.075,0.000001)(6.100,0.000001)(6.125,0.000001)(6.150,0.000001)(6.175,0.000001)(6.200,0.000001)(6.225,0.000001)(6.250,0.000001)(6.275,0.000001)(6.300,0.000001)(6.325,0.000001)(6.350,0.000001)(6.375,0.000001)(6.400,0.000001)(6.425,0.000001)(6.450,0.000001)(6.475,0.000001)(6.500,0.000001)(6.525,0.000001)(6.550,0.000001)(6.575,0.000001)(6.600,0.000001)(6.625,0.000001)(6.650,0.000001)(6.675,0.000001)(6.700,0.000001)(6.725,0.000001)(6.750,0.000001)(6.775,0.000001)(6.800,0.000001)(6.825,0.000001)(6.850,0.000001)(6.875,0.000001)(6.900,0.000001)(6.925,0.000001)(6.950,0.000001)(6.975,0.000001)(7.000,0.000001)(7.025,0.000001)(7.050,0.000001)(7.075,0.000001)(7.100,0.000001)(7.125,0.000001)(7.150,0.000001)(7.175,0.000001)(7.200,0.000001)(7.225,0.000001)(7.250,0.000001)(7.275,0.000001)(7.300,0.000001)(7.325,0.000001)(7.350,0.000001)(7.375,0.000001)(7.400,0.000001)(7.425,0.000001)(7.450,0.000001)(7.475,0.000001)(7.500,0.000001)(7.525,0.000001)(7.550,0.000001)(7.575,0.000001)(7.600,0.000001)(7.625,0.000001)(7.650,0.000001)(7.675,0.000001)(7.700,0.000001)(7.725,0.000001)(7.750,0.000001)(7.775,0.000001)(7.800,0.000001)(7.825,0.000001)(7.850,0.000001)(7.875,0.000001)(7.900,0.000001)(7.925,0.000001)(7.950,0.000001)(7.975,0.000001)(8.000,0.000001)
      };
      \legend{ {Eigenvalues of $\frac1nYY^*$},{Limiting spectral measure} }
    \end{axis}
  \end{tikzpicture}
  \caption{Histogram of the eigenvalues of $\frac1nYY^*$ against the limiting spectral measure, $L=2$, $p_1=p_2=1$, $N=200$, $n=1000$, Sudent-t impulsions.}
  \label{fig:histo_YY}
\end{figure}

\bigskip

Items 1.\@ and 2.\@ in Theorem~\ref{th:2} then provide a means to estimate $p_1,\ldots,p_{|\mathcal L|}$ and bilinear forms involving the eigenvectors of $AA^*$. In particular, if $p_k$ has multiplicity one in $\diag(p_1,\ldots,p_L)$, the summations in Item 2.\@ are irrelevant and we obtain an estimator for $a^*u_ku_k^*b$. These however explicitly rely on $\nu$ which, for practical purposes, might be of limited interest if the $\tau_i$ are statistically unknown. It turns out, from a careful understanding of $\gamma$, that
\begin{align*}
	\gamma - \hat{\gamma}_n \asto 0
\end{align*}
where
\begin{align}
	\label{eq:hatgamma}
	\hat{\gamma}_n &\triangleq \frac1n\sum_{i=1}^n \frac1Ny_i^*\hat{C}_{(i)}^{-1}y_i
\end{align}
and $\hat{C}_{(i)}=\hat{C}_N-\frac1n u(\frac1Ny_i^*\hat{C}_N^{-1}y_i)y_iy_i^*$.
Also, for any $M>0$,
\begin{align*}
	\max_{ \substack{1\leq j\leq n\\ \tau_j\leq M }} \left| \tau_j - \hat{\tau}_j\right| &\asto 0, \quad
	\max_{ \substack{1\leq j\leq n\\ \tau_j>M }} \left| 1 - \tau_j^{-1}\hat{\tau}_j\right| \asto 0
\end{align*}
where
\begin{align}
	\label{eq:hattau}
	\hat{\tau}_i \triangleq \frac1{\hat{\gamma}_n}\frac1Ny_i^*\hat{C}_{(i)}^{-1}y_i.
\end{align}
Details of these results are provided in Section~\ref{sec:proof}. Letting $\varepsilon>0$ small, for $x\in (S^++\varepsilon,\infty)$ and for all large $n$ a.s., we then denote $\hat{\delta}(x)$ the unique negative solution to\footnote{Remark here that, since $\hat{\tau}_i$, similar to $\tau_i$, may be found away from ${\rm Supp}(\nu)$, $\hat{\delta}(x)$ may not be defined everywhere in $(S_\mu^+,S^+)$ but is defined beyond $S^++\varepsilon$ for $n$ large a.s.}
\begin{align}
	\label{eq:hatdelta}
	\hat\delta(x) &= c_n \left( -x + \frac1n \sum_{i=1}^n \frac{\hat\tau_i v_c(\hat{\tau}_i\hat{\gamma}_n)}{1+\hat\delta(x)\hat{\tau}_iv_c(\hat{\tau}_i\hat{\gamma}_n)} \right)^{-1}.
\end{align}

From this, we then deduce the following alternative set of power and bilinear form estimators.

\begin{theorem}[Robust estimation for unknown $\nu$]
	\label{th:3}
	With the same notations as in Theorem~\ref{th:2}, and with $\hat{\gamma}_n$, $\hat{\tau_i}$, and $\hat{\delta}$ defined in \eqref{eq:hatgamma}--\eqref{eq:hatdelta}, we have the following results.

\medskip

{\sf 1. Purely empirical power estimation.} For each $j\in \mathcal L$,
\begin{align*}
	- \left( \hat{\delta}(\hat{\lambda}_j) \frac1N\sum_{i=1}^n \frac{v(\hat{\tau}_i\hat{\gamma}_n)}{1+\hat{\delta}(\hat{\lambda}_j)\hat{\tau}_iv(\hat{\tau}_i\hat{\gamma}_n)} \right)^{-1} \asto p_j.
\end{align*}

{\sf 2. Purely empirical bilinear form estimation.} For each $a,b\in\CC^N$ with $\Vert a\Vert=\Vert b\Vert=1$, and each $j\in\mathcal L$,
	\begin{align*}
		\sum_{k,p_k=p_j} a^*u_ku_k^* b - \sum_{k,p_k=p_j} \hat{w}_k a^* \hat{u}_k\hat{u}_k^* b  \asto 0
	\end{align*}
where
\begin{align*}
	\hat{w}_k &= \frac{\displaystyle \frac1n\sum_{i=1}^n \frac{v(\hat{\tau}_i\hat{\gamma}_n)}{\left( 1+\hat{\delta}(\hat{\lambda}_k)\hat{\tau}_i v(\hat{\tau}_i\hat{\gamma}_n) \right)^2} }{\displaystyle \frac1n\sum_{i=1}^n \frac{v(\hat{\tau}_i\hat{\gamma}_n)}{ 1+\hat{\delta}(\hat{\lambda}_k)\hat{\tau}_i v(\hat{\tau}_i\hat{\gamma}_n)} \left( 1 - \frac1N\sum_{i=1}^n \frac{\hat{\delta}(\hat{\lambda}_k)^2\tau_i^2v(\hat{\tau}_i\hat{\gamma}_n)^2}{\left( 1+\hat{\delta}(\hat{\lambda}_k)\hat{\tau}_i v(\hat{\tau}_i\hat{\gamma}_n) \right)^2} \right) }.
\end{align*}
\end{theorem}

Theorem~\ref{th:3} provides a means to estimate powers and bilinear forms without any statistical knowledge on the $\tau_i$, which are individually estimated. It is interesting to note that, since $\nu$ is only a limiting distribution, for practical systems, there is a priori no advantage in using the knowledge of $\nu$ or not. In particular, if $n$ is not too large in practice or if $\nu$ has heavy tails, it is highly probable that $\nu_n$ be quite distinct from $\nu$, leading the estimators in Theorem~\ref{th:1} to be likely less accurate than the estimators in Theorem~\ref{th:2}. Conversely, if $N$ is not too large, $\hat{\tau}_i$ may be a weak estimate for $\tau_i$ so that, if $\nu$ has much lighter tails, the estimators of Theorem~\ref{th:1} may have a better advantage. Theoretical performance comparison between both schemes would require to exhibit central limit theorems for these quantities, which we discuss in Section~\ref{sec:conclusion} but goes here beyond the scope of the present work.

\subsection{Application to angle estimation}

An important application of Theorem~\ref{th:1} and Theorem~\ref{th:2} is found in the context of array processing, briefly evoked in the second item of Remark~\ref{rem:examples}, in which $a_i=a(\theta_i)$ for some $\theta_i\in [0,2\pi)$. For theoretical convenience, we use the classical linear array representation for $a_i$ as follows.
\begin{assumption}
	\label{ass:arrayprocessing}
	For $i\in\{1,\ldots,L\}$, $a_i=a(\theta_i)$ with $\theta_1,\ldots,\theta_L$ distinct and, for $d>0$ and $\theta\in[0,2\pi)$,
	\begin{align*}
		a(\theta)=N^{-\frac12}[\exp(2\pi\imath dj\sin(\theta))]_{j=0}^{N-1}.
	\end{align*}
\end{assumption}
The objective in this specific model is to estimate $\theta_1,\ldots,\theta_L$ from the observations $y_1,\ldots,y_n$. In the regime $n\gg N$ with non-impulsive noise, this is efficiently performed by the traditional multiple signal classification (MUSIC) algorithm from \citep{SCH86}. Using the fact that the vectors $a(\theta_i)$, $i\in\{1,\ldots,L\}$, are orthogonal to the subspace spanned by the eigenvectors with eigenvalue $1$ of $\EE[y_1y_1^*]=AA^*+I_N$, the algorithm consists in retrieving the deepest minima of the nonnegative localization function $\hat\eta$ defined for $\theta\in [0,2\pi)$ by
\begin{align*}
	\hat\eta(\theta) &= a(\theta)^*\Pi_{\frac1nYY^*}a(\theta)
\end{align*}
where $\Pi_{\frac1nYY^*}$ is a projection matrix on the subspace associated with the $N-L$ smallest eigenvalues of $\frac1nYY^*$. Indeed, as $\frac1nYY^*$ is an almost surely consistent estimate for $\EE[y_1y_1^*]$ in the large $n$ regime, $\hat\eta(\theta)\asto \eta(\theta)$ where
\begin{align*}
	\eta(\theta) &= a(\theta)^*\Pi_{\EE[y_1y_1^*]} a(\theta)
\end{align*}
with here $\Pi_{\EE[y_1y_1^*]}$ a projection matrix on the subspace associated with the eigenvalue $1$ in $\EE[y_1y_1^*]$; as such, $\hat\eta(\theta)\asto 0$ for $\theta\in\{\theta_1,\ldots,\theta_L\}$ and to a positive quantity otherwise. In \citep{MES08}, Mestre proved that this algorithm is however inconsistent in the regime of Assumption~\ref{ass:c}. This led to \citep{MES08b} in which an improved estimator (the G-MUSIC estimator) for $\theta_1,\ldots,\theta_L$ was designed, however for a more involved model than the spiked model (i.e., $L$ is assumed commensurable with $N$). In \citep{LOU10b}, a spiked model hypothesis was then assumed (i.e., with $L$ small compared to $N,n$) which unfolded into a more practical and more theoretically tractable spiked G-MUSIC estimator. Similar to MUSIC, the latter consists in determining the deepest minima of an alternative localization function $\hat\eta_{\rm G}(\theta)$, which we shall define in a moment.

Although improved with respect to MUSIC, both algorithms still rely on exploiting the largest isolated eigenvalues of $\frac1nYY^*$ and the asymptotic boundedness of the noise spectrum. From the discussions in Section~\ref{sec:intro} and after Theorem~\ref{th:2}, under the generic Assumption~\ref{ass:y} with $\tau_i$ allowed to grow unbounded, these methods are now unreliable and in fact inefficient. From Item~2.\@ in both Theorem~\ref{th:2} and Theorem~\ref{th:3}, it is now possible to provide a consistent estimation method based on two novel localization functions $\hat\eta_{\rm RG}$ and $\hat\eta_{\rm RG}^{\rm emp}$. The resulting algorithms are from now on referred to as robust G-MUSIC and empirical robust G-MUSIC, respectively.

\begin{corollary}[Robust G-MUSIC]
	\label{co:RGMUSIC}
	Let Assumptions~\ref{ass:y}--\ref{ass:arrayprocessing} hold. Let $0<\kappa<\min_{i,j}|\theta_i-\theta_j|$ and denote $\mathcal R_i^\kappa=[\theta_i-\kappa/2,\theta_i+\kappa/2]$. Also define $\hat{\eta}_{\rm RG}(\theta)$ and $\hat{\eta}^{\rm emp}_{\rm RG}(\theta)$ as
	\begin{align*}
		\hat{\eta}_{\rm RG}(\theta) &= 1 - \sum_{k=1}^{|\mathcal L|} w_k a(\theta)^* \hat{u}_k\hat{u}_k a(\theta) \\
		\hat{\eta}^{\rm emp}_{\rm RG}(\theta) &= 1 - \sum_{k=1}^{|\mathcal L|} \hat{w}_k a(\theta)^* \hat{u}_k\hat{u}_k a(\theta)
	\end{align*}
	where we used the notations from Theorems~\ref{th:2} and \ref{th:3}. Then, for each $j\in\mathcal L$,
	\begin{align*}
		\hat{\theta}_j &\asto \theta_j \\
		\hat{\theta}^{\rm emp}_j &\asto \theta_j
	\end{align*}
	where
	\begin{align*}
		\hat{\theta}_j &\triangleq \argmin_{ \theta \in \mathcal R_j^\kappa } \left\{ \hat{\eta}_{\rm RG}(\theta) \right\} \\
		\hat{\theta}^{\rm emp}_j &\triangleq \argmin_{ \theta \in \mathcal R_j^\kappa } \left\{ \hat{\eta}^{\rm emp}_{\rm RG}(\theta) \right\}.
	\end{align*}
\end{corollary}

With the same reasoning as in Remark~\ref{rem:v=1}, it is now easy to check that, letting the $v_c$ or $v$ functions be replaced by the constant~$1$ in the expressions of $w_k$ and $\hat{w}_k$, respectively, we fall back on G-MUSIC schemes devised in e.g., \citep{LOU10b}. In what follows, we then define $\hat\eta_{\rm G}(\theta)$ and $\hat{\eta}_{\rm G}^{\rm emp}(\theta)$ similarly to $\hat\eta_{\rm RG}(\theta)$ and $\hat{\eta}_{\rm RG}^{\rm emp}(\theta)$ but with the functions $v_c$ and $v$ replaced by the constant~$1$ and with the couples $(\hat{\lambda}_k,\hat{u}_k)$ replaced by the $k$-th largest eigenvalue and associated eigenvectors of $\frac1nYY^*$. For a further comparison of the various methods, we also denote by $\hat\eta_{\rm R}(\theta)$ the robust counterpart to $\hat{\eta}(\theta)$ defined by $\hat\eta_{\rm R}(\theta)=a(\theta)^*\Pi_{\hat{C}_N} a(\theta)$ with $\Pi_{\hat{C}_N}$ a projection matrix on the subspace associated with the $N-L$ smallest eigenvalues of $\hat{C}_N$.

Simulation curves are provided below which compare the performance of the various improved MUSIC techniques. Since the methods based on the extraction of $\delta(\hat{\lambda}_i)$ may be void when this value does not exist, we blindly proceed by solving the fixed-point equation defining $\delta(\hat{\lambda}_i)$ thanks to the standard fixed-point algorithm until convergence or until a maximum number of iterations is reached. This effect is in fact marginal as it is theoretically highly probable that eigenvalues be found beyond $S_\mu^+$ for each finite $N,n$. We also assume $\mathcal L=\{1,\ldots,L\}$ even if this does not hold, which in practice one cannot anticipate. Voluntarily disrupting from the theoretical claims of Theorems~\ref{th:1}--\ref{th:3} will allow for an observation of problems arising when the assumptions are not fully satisfied. In all simulation figures, we consider $u(x)=(1+\alpha)(\alpha+x)^{-1}$ with $\alpha=0.2$, $N=20$, $n=100$, $L=2$, $\theta_1=10^\circ$, $\theta_2=12^\circ$. The noise impulsions are of two types: (i) single outlier impulsion for which $\tau_i=1$, $i\in\{1,\ldots,n-1\}$ and $\tau_n=100$, or (ii) Student impulsions for which $\tau_i=t^2(\beta-2)\beta^{-1}$ with $t$ a Student-t random variable with parameter $\beta=100$ (the normalization ensures $\EE[\tau_1]=1$). 

\begin{figure}[h!]
  \centering
  \begin{tikzpicture}[font=\footnotesize]
    \renewcommand{\axisdefaulttryminticks}{4} 
    \tikzstyle{every major grid}+=[style=densely dashed]       
    \tikzstyle{every axis legend}+=[cells={anchor=west},fill=white,
        at={(1.02,1.00)}, anchor=north west, font=\scriptsize ]
    \begin{axis}[
      xmin=4,
      ymin=-.15,
      xmax=18,
      ymax=1,
      grid=none,
      ymajorgrids=false,
      scaled ticks=true,
      mark repeat=5,
      mark options=solid,
      mark size=1.5pt,
      xlabel={$\theta$ [deg]},
      ylabel={Localization functions $\hat{\eta}_X(\theta)$}
      ]
      \addplot[black,mark=*,line width=0.5pt] plot coordinates{
	      (4.01,0.992486)(4.11,0.990594)(4.21,0.987632)(4.31,0.983647)(4.41,0.978734)(4.51,0.973034)(4.61,0.966728)(4.71,0.960035)(4.81,0.953201)(4.91,0.946490)(5.01,0.940175)(5.11,0.934526)(5.21,0.929796)(5.31,0.926211)(5.41,0.923958)(5.51,0.923173)(5.61,0.923931)(5.71,0.926241)(5.81,0.930036)(5.91,0.935173)(6.01,0.941434)(6.11,0.948521)(6.21,0.956071)(6.31,0.963656)(6.41,0.970798)(6.51,0.976978)(6.61,0.981657)(6.71,0.984285)(6.81,0.984324)(6.91,0.981264)(7.01,0.974639)(7.11,0.964046)(7.21,0.949160)(7.31,0.929745)(7.41,0.905668)(7.51,0.876908)(7.61,0.843556)(7.71,0.805818)(7.81,0.764016)(7.91,0.718579)(8.01,0.670034)(8.11,0.618991)(8.21,0.566133)(8.31,0.512194)(8.41,0.457937)(8.51,0.404139)(8.61,0.351566)(8.71,0.300951)(8.81,0.252973)(8.91,0.208243)(9.01,0.167280)(9.11,0.130503)(9.21,0.098216)(9.31,0.070606)(9.41,0.047732)(9.51,0.029534)(9.61,0.015833)(9.71,0.006339)(9.81,0.000667)(9.91,-0.001653)(10.01,-0.001155)(10.11,0.001584)(10.21,0.005962)(10.31,0.011374)(10.41,0.017235)(10.51,0.023002)(10.61,0.028189)(10.71,0.032389)(10.81,0.035285)(10.91,0.036662)(11.01,0.036413)(11.11,0.034546)(11.21,0.031181)(11.31,0.026545)(11.41,0.020967)(11.51,0.014861)(11.61,0.008719)(11.71,0.003085)(11.81,-0.001456)(11.91,-0.004306)(12.01,-0.004870)(12.11,-0.002580)(12.21,0.003087)(12.31,0.012586)(12.41,0.026291)(12.51,0.044480)(12.61,0.067321)(12.71,0.094870)(12.81,0.127059)(12.91,0.163704)(13.01,0.204502)(13.11,0.249043)(13.21,0.296816)(13.31,0.347227)(13.41,0.399611)(13.51,0.453252)(13.61,0.507405)(13.71,0.561310)(13.81,0.614217)(13.91,0.665406)(14.01,0.714204)(14.11,0.760002)(14.21,0.802271)(14.31,0.840572)(14.41,0.874570)(14.51,0.904036)(14.61,0.928849)(14.71,0.949002)(14.81,0.964590)(14.91,0.975809)(15.01,0.982947)(15.11,0.986369)(15.21,0.986504)(15.31,0.983835)(15.41,0.978876)(15.51,0.972161)(15.61,0.964225)(15.71,0.955590)(15.81,0.946750)(15.91,0.938159)(16.01,0.930219)(16.11,0.923272)(16.21,0.917593)(16.31,0.913386)(16.41,0.910781)(16.51,0.909838)(16.61,0.910548)(16.71,0.912837)(16.81,0.916577)(16.91,0.921590)(17.01,0.927659)(17.11,0.934539)(17.21,0.941968)(17.31,0.949677)(17.41,0.957400)(17.51,0.964884)(17.61,0.971898)(17.71,0.978240)(17.81,0.983742)(17.91,0.988276)(18.01,0.991756)	      
      };
      \addplot[black,densely dashed,mark=*,line width=0.5pt] plot coordinates{
(4.01,0.992466)(4.11,0.990562)(4.21,0.987582)(4.31,0.983573)(4.41,0.978631)(4.51,0.972897)(4.61,0.966554)(4.71,0.959823)(4.81,0.952949)(4.91,0.946200)(5.01,0.939850)(5.11,0.934169)(5.21,0.929414)(5.31,0.925811)(5.41,0.923548)(5.51,0.922762)(5.61,0.923528)(5.71,0.925855)(5.81,0.929675)(5.91,0.934846)(6.01,0.941146)(6.11,0.948277)(6.21,0.955873)(6.31,0.963504)(6.41,0.970687)(6.51,0.976904)(6.61,0.981608)(6.71,0.984250)(6.81,0.984287)(6.91,0.981206)(7.01,0.974540)(7.11,0.963883)(7.21,0.948908)(7.31,0.929380)(7.41,0.905166)(7.51,0.876243)(7.61,0.842704)(7.71,0.804760)(7.81,0.762732)(7.91,0.717055)(8.01,0.668257)(8.11,0.616955)(8.21,0.563836)(8.31,0.509638)(8.41,0.455130)(8.51,0.401095)(8.61,0.348301)(8.71,0.297487)(8.81,0.249337)(8.91,0.204464)(9.01,0.163390)(9.11,0.126535)(9.21,0.094205)(9.31,0.066586)(9.41,0.043739)(9.51,0.025599)(9.61,0.011985)(9.71,0.002605)(9.81,-0.002932)(9.91,-0.005099)(10.01,-0.004435)(10.11,-0.001523)(10.21,0.003029)(10.31,0.008611)(10.41,0.014633)(10.51,0.020547)(10.61,0.025863)(10.71,0.030170)(10.81,0.033147)(10.91,0.034578)(11.01,0.034354)(11.11,0.032484)(11.21,0.029087)(11.31,0.024392)(11.41,0.018730)(11.51,0.012519)(11.61,0.006254)(11.71,0.000484)(11.81,-0.004202)(11.91,-0.007200)(12.01,-0.007910)(12.11,-0.005759)(12.21,-0.000220)(12.31,0.009168)(12.41,0.022783)(12.51,0.040906)(12.61,0.063709)(12.71,0.091248)(12.81,0.123459)(12.91,0.160156)(13.01,0.201038)(13.11,0.245692)(13.21,0.293607)(13.31,0.344184)(13.41,0.396756)(13.51,0.450604)(13.61,0.504977)(13.71,0.559112)(13.81,0.612255)(13.91,0.663680)(14.01,0.712710)(14.11,0.758732)(14.21,0.801213)(14.31,0.839711)(14.41,0.873889)(14.51,0.903513)(14.61,0.928464)(14.71,0.948731)(14.81,0.964410)(14.91,0.975698)(15.01,0.982883)(15.11,0.986331)(15.21,0.986474)(15.31,0.983797)(15.41,0.978816)(15.51,0.972067)(15.61,0.964090)(15.71,0.955408)(15.81,0.946519)(15.91,0.937879)(16.01,0.929892)(16.11,0.922903)(16.21,0.917188)(16.31,0.912952)(16.41,0.910328)(16.51,0.909374)(16.61,0.910083)(16.71,0.912380)(16.81,0.916136)(16.91,0.921172)(17.01,0.927270)(17.11,0.934185)(17.21,0.941654)(17.31,0.949404)(17.41,0.957169)(17.51,0.964694)(17.61,0.971747)(17.71,0.978125)(17.81,0.983658)(17.91,0.988219)(18.01,0.991719)
      };
      \addplot[black,mark=o,line width=0.5pt] plot coordinates{
	      (4.01,0.978274)(4.11,0.979868)(4.21,0.981527)(4.31,0.983166)(4.41,0.984696)(4.51,0.986029)(4.61,0.987081)(4.71,0.987777)(4.81,0.988051)(4.91,0.987856)(5.01,0.987161)(5.11,0.985955)(5.21,0.984250)(5.31,0.982081)(5.41,0.979505)(5.51,0.976602)(5.61,0.973471)(5.71,0.970226)(5.81,0.966996)(5.91,0.963918)(6.01,0.961132)(6.11,0.958773)(6.21,0.956971)(6.31,0.955836)(6.41,0.955461)(6.51,0.955909)(6.61,0.957212)(6.71,0.959363)(6.81,0.962316)(6.91,0.965981)(7.01,0.970223)(7.11,0.974864)(7.21,0.979682)(7.31,0.984415)(7.41,0.988764)(7.51,0.992404)(7.61,0.994984)(7.71,0.996139)(7.81,0.995500)(7.91,0.992700)(8.01,0.987389)(8.11,0.979240)(8.21,0.967961)(8.31,0.953306)(8.41,0.935082)(8.51,0.913157)(8.61,0.887470)(8.71,0.858033)(8.81,0.824934)(8.91,0.788345)(9.01,0.748512)(9.11,0.705765)(9.21,0.660502)(9.31,0.613193)(9.41,0.564367)(9.51,0.514605)(9.61,0.464528)(9.71,0.414787)(9.81,0.366051)(9.91,0.318990)(10.01,0.274265)(10.11,0.232515)(10.21,0.194343)(10.31,0.160303)(10.41,0.130891)(10.51,0.106531)(10.61,0.087574)(10.71,0.074282)(10.81,0.066833)(10.91,0.065308)(11.01,0.069700)(11.11,0.079906)(11.21,0.095738)(11.31,0.116923)(11.41,0.143109)(11.51,0.173877)(11.61,0.208749)(11.71,0.247194)(11.81,0.288647)(11.91,0.332512)(12.01,0.378183)(12.11,0.425047)(12.21,0.472501)(12.31,0.519961)(12.41,0.566872)(12.51,0.612718)(12.61,0.657027)(12.71,0.699380)(12.81,0.739416)(12.91,0.776835)(13.01,0.811399)(13.11,0.842934)(13.21,0.871328)(13.31,0.896530)(13.41,0.918548)(13.51,0.937440)(13.61,0.953315)(13.71,0.966321)(13.81,0.976644)(13.91,0.984498)(14.01,0.990118)(14.11,0.993755)(14.21,0.995670)(14.31,0.996124)(14.41,0.995375)(14.51,0.993674)(14.61,0.991256)(14.71,0.988342)(14.81,0.985129)(14.91,0.981795)(15.01,0.978493)(15.11,0.975352)(15.21,0.972474)(15.31,0.969939)(15.41,0.967804)(15.51,0.966105)(15.61,0.964858)(15.71,0.964061)(15.81,0.963700)(15.91,0.963747)(16.01,0.964166)(16.11,0.964913)(16.21,0.965942)(16.31,0.967201)(16.41,0.968643)(16.51,0.970218)(16.61,0.971882)(16.71,0.973595)(16.81,0.975319)(16.91,0.977026)(17.01,0.978691)(17.11,0.980294)(17.21,0.981823)(17.31,0.983270)(17.41,0.984629)(17.51,0.985903)(17.61,0.987092)(17.71,0.988202)(17.81,0.989240)(17.91,0.990211)(18.01,0.991123)	      
      };
      \addplot[black,densely dashed,mark=o,line width=0.5pt] plot coordinates{
	      (4.01,0.972232)(4.11,0.970874)(4.21,0.966831)(4.31,0.960088)(4.41,0.950731)(4.51,0.938948)(4.61,0.925026)(4.71,0.909342)(4.81,0.892357)(4.91,0.874596)(5.01,0.856634)(5.11,0.839077)(5.21,0.822537)(5.31,0.807615)(5.41,0.794872)(5.51,0.784813)(5.61,0.777859)(5.71,0.774335)(5.81,0.774450)(5.91,0.778287)(6.01,0.785792)(6.11,0.796772)(6.21,0.810897)(6.31,0.827703)(6.41,0.846603)(6.51,0.866905)(6.61,0.887824)(6.71,0.908513)(6.81,0.928079)(6.91,0.945617)(7.01,0.960234)(7.11,0.971078)(7.21,0.977366)(7.31,0.978409)(7.41,0.973634)(7.51,0.962608)(7.61,0.945046)(7.71,0.920828)(7.81,0.890002)(7.91,0.852788)(8.01,0.809566)(8.11,0.760875)(8.21,0.707392)(8.31,0.649916)(8.41,0.589342)(8.51,0.526641)(8.61,0.462826)(8.71,0.398926)(8.81,0.335959)(8.91,0.274898)(9.01,0.216648)(9.11,0.162020)(9.21,0.111713)(9.31,0.066294)(9.41,0.026186)(9.51,-0.008338)(9.61,-0.037157)(9.71,-0.060304)(9.81,-0.077950)(9.91,-0.090400)(10.01,-0.098074)(10.11,-0.101483)(10.21,-0.101214)(10.31,-0.097900)(10.41,-0.092197)(10.51,-0.084757)(10.61,-0.076207)(10.71,-0.067122)(10.81,-0.058006)(10.91,-0.049276)(11.01,-0.041245)(11.11,-0.034117)(11.21,-0.027978)(11.31,-0.022798)(11.41,-0.018435)(11.51,-0.014642)(11.61,-0.011079)(11.71,-0.007330)(11.81,-0.002922)(11.91,0.002657)(12.01,0.009934)(12.11,0.019427)(12.21,0.031627)(12.31,0.046974)(12.41,0.065837)(12.51,0.088500)(12.61,0.115146)(12.71,0.145848)(12.81,0.180561)(12.91,0.219124)(13.01,0.261257)(13.11,0.306571)(13.21,0.354573)(13.31,0.404687)(13.41,0.456262)(13.51,0.508595)(13.61,0.560950)(13.71,0.612583)(13.81,0.662755)(13.91,0.710762)(14.01,0.755950)(14.11,0.797736)(14.21,0.835621)(14.31,0.869205)(14.41,0.898196)(14.51,0.922417)(14.61,0.941808)(14.71,0.956426)(14.81,0.966438)(14.91,0.972117)(15.01,0.973827)(15.11,0.972013)(15.21,0.967183)(15.31,0.959893)(15.41,0.950728)(15.51,0.940282)(15.61,0.929146)(15.71,0.917882)(15.81,0.907017)(15.91,0.897021)(16.01,0.888300)(16.11,0.881184)(16.21,0.875925)(16.31,0.872690)(16.41,0.871560)(16.51,0.872535)(16.61,0.875537)(16.71,0.880417)(16.81,0.886967)(16.91,0.894924)(17.01,0.903991)(17.11,0.913843)(17.21,0.924141)(17.31,0.934549)(17.41,0.944739)(17.51,0.954410)(17.61,0.963289)(17.71,0.971147)(17.81,0.977803)(17.91,0.983124)(18.01,0.987032)	      
      };
      \addplot[black,mark=triangle,line width=0.5pt] plot coordinates{
	      (4.01,0.998748)(4.11,0.996434)(4.21,0.992939)(4.31,0.988338)(4.41,0.982754)(4.51,0.976364)(4.61,0.969388)(4.71,0.962084)(4.81,0.954739)(4.91,0.947657)(5.01,0.941145)(5.11,0.935503)(5.21,0.931004)(5.31,0.927888)(5.41,0.926341)(5.51,0.926490)(5.61,0.928387)(5.71,0.932007)(5.81,0.937236)(5.91,0.943877)(6.01,0.951646)(6.11,0.960178)(6.21,0.969036)(6.31,0.977721)(6.41,0.985687)(6.51,0.992360)(6.61,0.997155)(6.71,0.999495)(6.81,0.998836)(6.91,0.994686)(7.01,0.986627)(7.11,0.974332)(7.21,0.957586)(7.31,0.936297)(7.41,0.910504)(7.51,0.880392)(7.61,0.846286)(7.71,0.808651)(7.81,0.768087)(7.91,0.725315)(8.01,0.681159)(8.11,0.636531)(8.21,0.592401)(8.31,0.549776)(8.41,0.509668)(8.51,0.473063)(8.61,0.440898)(8.71,0.414023)(8.81,0.393181)(8.91,0.378977)(9.01,0.371864)(9.11,0.372118)(9.21,0.379833)(9.31,0.394909)(9.41,0.417054)(9.51,0.445786)(9.61,0.480447)(9.71,0.520213)(9.81,0.564122)(9.91,0.611092)(10.01,0.659956)(10.11,0.709490)(10.21,0.758447)(10.31,0.805595)(10.41,0.849746)(10.51,0.889794)(10.61,0.924745)(10.71,0.953740)(10.81,0.976088)(10.91,0.991277)(11.01,0.998989)(11.11,0.999113)(11.21,0.991740)(11.31,0.977166)(11.41,0.955876)(11.51,0.928534)(11.61,0.895957)(11.71,0.859097)(11.81,0.819009)(11.91,0.776821)(12.01,0.733700)(12.11,0.690822)(12.21,0.649337)(12.31,0.610339)(12.41,0.574835)(12.51,0.543719)(12.61,0.517751)(12.71,0.497538)(12.81,0.483521)(12.91,0.475966)(13.01,0.474963)(13.11,0.480428)(13.21,0.492109)(13.31,0.509602)(13.41,0.532362)(13.51,0.559727)(13.61,0.590940)(13.71,0.625173)(13.81,0.661555)(13.91,0.699194)(14.01,0.737209)(14.11,0.774750)(14.21,0.811025)(14.31,0.845316)(14.41,0.876999)(14.51,0.905557)(14.61,0.930590)(14.71,0.951820)(14.81,0.969092)(14.91,0.982375)(15.01,0.991753)(15.11,0.997418)(15.21,0.999655)(15.31,0.998830)(15.41,0.995375)(15.51,0.989766)(15.61,0.982509)(15.71,0.974120)(15.81,0.965109)(15.91,0.955962)(16.01,0.947128)(16.11,0.939005)(16.21,0.931932)(16.31,0.926180)(16.41,0.921948)(16.51,0.919360)(16.61,0.918465)(16.71,0.919243)(16.81,0.921609)(16.91,0.925419)(17.01,0.930481)(17.11,0.936564)(17.21,0.943411)(17.31,0.950748)(17.41,0.958295)(17.51,0.965780)(17.61,0.972947)(17.71,0.979564)(17.81,0.985432)(17.91,0.990392)(18.01,0.994324)	      
      };
      \addplot[black,mark=square,line width=0.5pt] plot coordinates{
(4.01,0.996788)(4.11,0.994221)(4.21,0.989381)(4.31,0.982331)(4.41,0.973222)(4.51,0.962291)(4.61,0.949849)(4.71,0.936278)(4.81,0.922013)(4.91,0.907535)(5.01,0.893344)(5.11,0.879950)(5.21,0.867848)(5.31,0.857502)(5.41,0.849323)(5.51,0.843655)(5.61,0.840757)(5.71,0.840791)(5.81,0.843815)(5.91,0.849772)(6.01,0.858491)(6.11,0.869690)(6.21,0.882979)(6.31,0.897875)(6.41,0.913811)(6.51,0.930156)(6.61,0.946237)(6.71,0.961358)(6.81,0.974825)(6.91,0.985974)(7.01,0.994188)(7.11,0.998927)(7.21,0.999747)(7.31,0.996316)(7.41,0.988433)(7.51,0.976034)(7.61,0.959203)(7.71,0.938172)(7.81,0.913316)(7.91,0.885147)(8.01,0.854299)(8.11,0.821511)(8.21,0.787607)(8.31,0.753469)(8.41,0.720012)(8.51,0.688155)(8.61,0.658790)(8.71,0.632756)(8.81,0.610811)(8.91,0.593605)(9.01,0.581658)(9.11,0.575345)(9.21,0.574877)(9.31,0.580297)(9.41,0.591476)(9.51,0.608111)(9.61,0.629738)(9.71,0.655742)(9.81,0.685374)(9.91,0.717774)(10.01,0.751997)(10.11,0.787040)(10.21,0.821875)(10.31,0.855476)(10.41,0.886853)(10.51,0.915080)(10.61,0.939324)(10.71,0.958873)(10.81,0.973150)(10.91,0.981736)(11.01,0.984381)(11.11,0.981008)(11.21,0.971718)(11.31,0.956782)(11.41,0.936638)(11.51,0.911868)(11.61,0.883188)(11.71,0.851420)(11.81,0.817470)(11.91,0.782298)(12.01,0.746890)(12.11,0.712228)(12.21,0.679266)(12.31,0.648894)(12.41,0.621922)(12.51,0.599051)(12.61,0.580861)(12.71,0.567790)(12.81,0.560130)(12.91,0.558020)(13.01,0.561445)(13.11,0.570240)(13.21,0.584103)(13.31,0.602601)(13.41,0.625193)(13.51,0.651245)(13.61,0.680052)(13.71,0.710862)(13.81,0.742901)(13.91,0.775393)(14.01,0.807584)(14.11,0.838765)(14.21,0.868289)(14.31,0.895589)(14.41,0.920190)(14.51,0.941718)(14.61,0.959911)(14.71,0.974616)(14.81,0.985791)(14.91,0.993500)(15.01,0.997905)(15.11,0.999257)(15.21,0.997881)(15.31,0.994164)(15.41,0.988537)(15.51,0.981458)(15.61,0.973400)(15.71,0.964829)(15.81,0.956191)(15.91,0.947900)(16.01,0.940324)(16.11,0.933775)(16.21,0.928503)(16.31,0.924691)(16.41,0.922452)(16.51,0.921828)(16.61,0.922796)(16.71,0.925270)(16.81,0.929110)(16.91,0.934129)(17.01,0.940101)(17.11,0.946776)(17.21,0.953888)(17.31,0.961164)(17.41,0.968339)(17.51,0.975164)(17.61,0.981412)(17.71,0.986890)(17.81,0.991441)(17.91,0.994951)(18.01,0.997349)
      };
      \addplot[gray,densely dashed] plot coordinates{
        (4,0)(18,0)
      };
      \addplot[gray,densely dashed] plot coordinates{
        (12,-1)(12,1)
      };
      \addplot[gray,densely dashed] plot coordinates{
        (10,-1)(10,1)
      };
      \legend{ {Robust G-MUSIC}, {Emp. robust G-MUSIC}, {G-MUSIC},{Emp. G-MUSIC},{Robust MUSIC},{MUSIC} }
    \end{axis}
  \end{tikzpicture}
  \caption{Random realization of the localization functions for the various MUSIC estimators, with $N=20$, $n=100$, two sources at $10^\circ$ and $12^\circ$, Student-t impulsions with parameter $\beta=100$, $u(x)=(1+\alpha)/(\alpha+x)$ with $\alpha=0.2$. Powers $p_1=p_2=10^{0.5}=5~{\rm dB}$.}
  \label{fig:oneshot}
\end{figure}
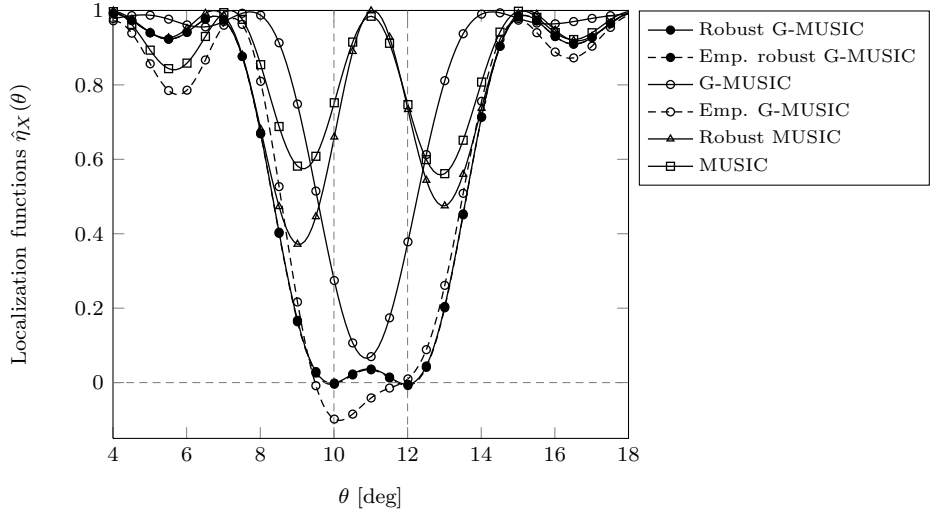

Figure~\ref{fig:oneshot} provides a single realization (but representative of the multiple realizations we simulated) of the various localization functions $\hat{\eta}_X$ and $\hat{\eta}_X^{\rm emp}$ for $\theta$ in the vicinity of $\theta_1,\theta_2$, $X$ being void, $R$, or $RG$. The scenario considered is that of a Student-t noise and $p_1=p_2=1$. The figure confirms the advantage of the methods based on $\hat{C}_N$ over $\frac1nYY^*$ which unfolds from the proper extreme eigenvalue isolation observed under the same setting in Figure~\ref{fig:histo_C} against Figure~\ref{fig:histo_YY}. Due to $N/n$ being non trivial, while the robust G-MUSIC methods accurately discriminate both angles at their precise locations and with appropriate localization function amplitude, the robust MUSIC approach discriminates the two angles at erroneous locations and erroneous localization function amplitude. Benefiting from the random matrix advantage, G-MUSIC in turn behaves better in amplitude than MUSIC but cannot discriminate angles. Observe also here that both empirical and non-empirical robust G-MUSIC approaches behave extremely similar (both curves are visually superimposed), suggesting that with $\beta=100$ the samples from the Student-t distribution represent sufficiently well the actual distribution of $\tau_1v(\tau_1\gamma)$. This no longer holds for G-MUSIC versus empirical G-MUSIC, in which case the approximation of $\nu_n$ by the distribution $\nu$ of $\tau_1$ is not appropriate.

Figure~\ref{fig:perf_Stud} and Figure~\ref{fig:perf_Spike} provide the mean square error performance for the first angle estimation $\EE[|\hat{\theta}_1-\theta_1|^2]$ as a function of the source powers $p_1=p_2$; the estimates are based for each estimator on retrieving the local minima of $\hat{\eta}_X$. For fair comparison, the two deepest minima of the localization functions are extracted and $\hat{\theta}_1$ is declared to be the estimated angle closest to $\theta_1$ (in particular, if a unique minimum is found close to any $\theta_i$, $\hat{\theta}_1$ is attached to this minimum). Figure~\ref{fig:perf_Stud} assumes the Student-t impulsion scenario of Figure~\ref{fig:oneshot}, while Figure~\ref{fig:perf_Spike} is concerned with the outlier impulsion model previously described. Both figures further confirm the advantage brought by the robust G-MUSIC scheme with asymptotic equivalence between empirical or non-empirical in the large source power regime. We observe in particular the outstanding advantage of (robust or not) G-MUSIC methods which perform well at high source power, while standard methods saturate. Interestingly, from Figure~\ref{fig:perf_Stud}, the G-MUSIC schemes perform well in the high source power regime, which corresponds to scenarios in which the noise impulsion amplitudes are often small enough compared to source power to be assumed bounded and G-MUSIC is then consistent. Nonetheless, G-MUSIC never closes the gap with robust G-MUSIC which is likely explained by the much larger spacing between noise and information eigenvalues in the spectrum of $\hat{C}_N$. The situation is different in Figure~\ref{fig:perf_Spike} where G-MUSIC almost meets the performance of robust G-MUSIC at very high power, while performing poorly below $20~{\rm dB}$. This is explained by the presence of a single additional eigenvalue of amplitude around $100$ (i.e., $20$~dB) in the spectrum of $\frac1nYY^*$ which corrupts the G-MUSIC algorithm as long as this amplitude is larger than these of the two informative eigenvalues due to the steering vectors (about $p_1$).

\begin{figure}[h!]
  \centering
  \begin{tikzpicture}[font=\footnotesize]
    \renewcommand{\axisdefaulttryminticks}{4} 
    \tikzstyle{every major grid}+=[style=densely dashed]       
    \tikzstyle{every axis y label}+=[yshift=-10pt] 
    \tikzstyle{every axis x label}+=[yshift=5pt]
    \tikzstyle{every axis legend}+=[cells={anchor=west},fill=white,
        at={(1.02,1.00)}, anchor=north west, font=\scriptsize ]
    \begin{semilogyaxis}[
      xmin=-5,
      xmax=30,
      ymax=.1,
      grid=major,
      ymajorgrids=false,
      scaled ticks=true,
      mark options=solid,
      mark size=1.5pt,
      xlabel={$p_1,p_2$ [dB]},
      ylabel={Mean square error $\EE[|\hat{\theta}_1-\theta_1|^2]$}
      ]
      \addplot[black,smooth,mark=*,line width=0.5pt] plot coordinates{
	      (-5,0.0123206704)(-2.500000e+00,0.0003091066)(0,0.0000739683)(2.500000e+00,0.0000421487)(5,0.0000190099)(7.500000e+00,0.0000089245)(10,0.0000044831)(1.250000e+01,0.0000022308)(15,0.0000011068)(1.750000e+01,0.0000006504)(20,0.0000004055)(2.250000e+01,0.0000002394)(25,0.0000001317)(2.750000e+01,0.0000000736)(30,0.0000000412)	      
      };
      \addplot[black,smooth,densely dashed,mark=*,line width=0.5pt] plot coordinates{
(-5,0.0114846764)(-2.500000e+00,0.0009542529)(0,0.0001646572)(2.500000e+00,0.0000803358)(5,0.0000322975)(7.500000e+00,0.0000127943)(10,0.0000060773)(1.250000e+01,0.0000027184)(15,0.0000011202)(1.750000e+01,0.0000006604)(20,0.0000004078)(2.250000e+01,0.0000002400)(25,0.0000001325)(2.750000e+01,0.0000000737)(30,0.0000000412)
      };
      \addplot[black,smooth,mark=o,line width=0.5pt] plot coordinates{
(-5,0.1079380734)(-2.500000e+00,0.0613164315)(0,0.0188549991)(2.500000e+00,0.0119959252)(5,0.0108074882)(7.500000e+00,0.0045589191)(10,0.0002498616)(1.250000e+01,0.0000454129)(15,0.0000182000)(1.750000e+01,0.0000087839)(20,0.0000045024)(2.250000e+01,0.0000021508)(25,0.0000011034)(2.750000e+01,0.0000005900)(30,0.0000002866)
      };
      \addplot[black,smooth,densely dashed,mark=o,line width=0.5pt] plot coordinates{
(-5,0.1040418047)(-2.500000e+00,0.0480746033)(0,0.0157980302)(2.500000e+00,0.0114400362)(5,0.0106193965)(7.500000e+00,0.0045236076)(10,0.0000972674)(1.250000e+01,0.0000428841)(15,0.0000182238)(1.750000e+01,0.0000088085)(20,0.0000048011)(2.250000e+01,0.0000021286)(25,0.0000010945)(2.750000e+01,0.0000005874)(30,0.0000002883)
      };
      \addplot[black,smooth,mark=triangle,line width=0.5pt] plot coordinates{
(-5,0.040408)(-2.500000e+00,0.009271)(0,0.009180)(2.500000e+00,0.009258)(5,0.009190)(7.500000e+00,0.009199)(10,0.009178)(1.250000e+01,0.009207)(15,0.009174)(1.750000e+01,0.009173)(20,0.009177)(2.250000e+01,0.009184)(25,0.009179)(2.750000e+01,0.009205)(30,0.009177)
      };
      \addplot[black,smooth,mark=square,line width=0.5pt] plot coordinates{
(-5,0.136500)(-2.500000e+00,0.075386)(0,0.052649)(2.500000e+00,0.048044)(5,0.012180)(7.500000e+00,0.009254)(10,0.009213)(1.250000e+01,0.009201)(15,0.009172)(1.750000e+01,0.009169)(20,0.009164)(2.250000e+01,0.009182)(25,0.009180)(2.750000e+01,0.009195)(30,0.009163)
      };
      \legend{ {Robust G-MUSIC}, {Emp. robust G-MUSIC}, {G-MUSIC},{Emp. G-MUSIC},{Robust MUSIC},{MUSIC} }
    \end{semilogyaxis}
  \end{tikzpicture}
  \caption{Means square error performance of the estimation of $\theta_1=10^\circ$, with $N=20$, $n=100$, two sources at $10^\circ$ and $12^\circ$, Student-t impulsions with parameter $\beta=10$, $u(x)=(1+\alpha)/(\alpha+x)$ with $\alpha=0.2$, $p_1=p_2$.}
  \label{fig:perf_Stud}
\end{figure}
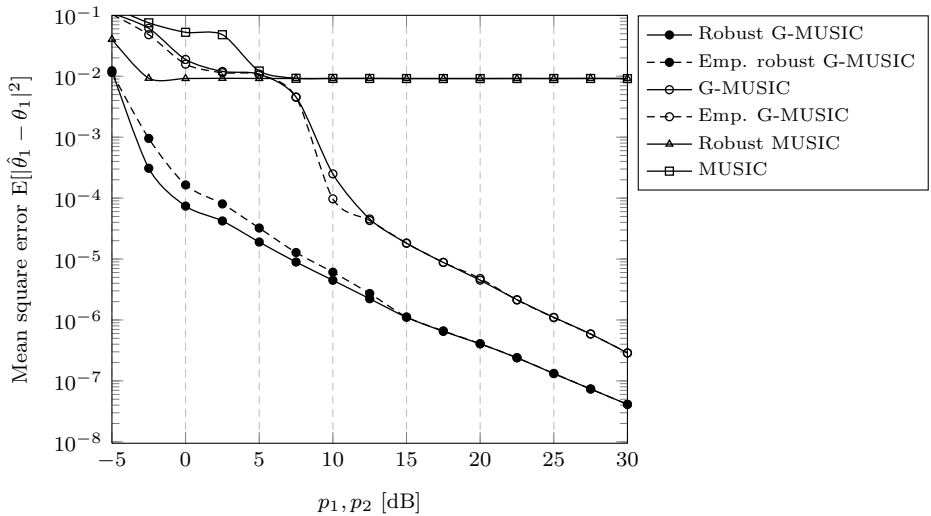

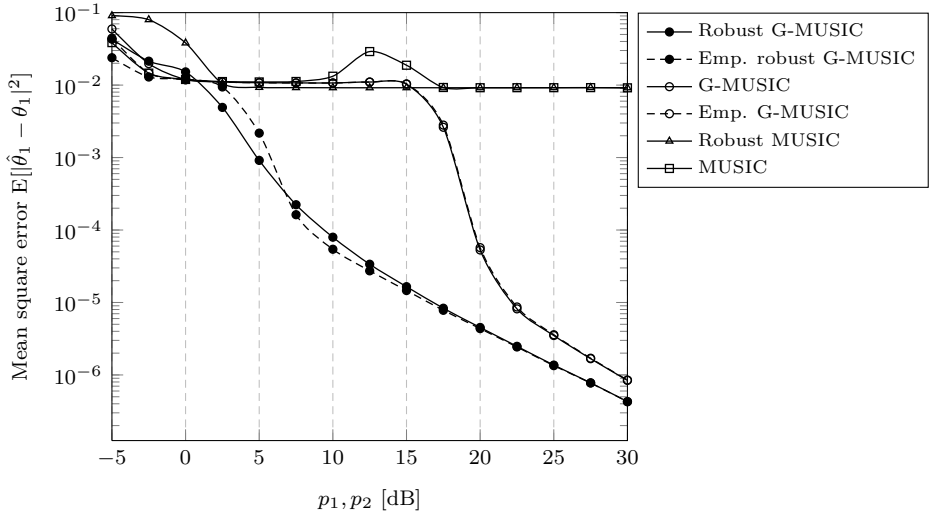
\begin{figure}[h!]
  \centering
  \begin{tikzpicture}[font=\footnotesize]
    \renewcommand{\axisdefaulttryminticks}{4} 
    \tikzstyle{every major grid}+=[style=densely dashed]       
    \tikzstyle{every axis y label}+=[yshift=-10pt] 
    \tikzstyle{every axis x label}+=[yshift=5pt]
    \tikzstyle{every axis legend}+=[cells={anchor=west},fill=white,
        at={(1.02,1.00)}, anchor=north west, font=\scriptsize ]
    \begin{semilogyaxis}[
      xmin=-5,
      xmax=30,
      ymax=.1,
      grid=major,
      ymajorgrids=false,
      scaled ticks=true,
      mark options=solid,
      mark size=1.5pt,
      xlabel={$p_1,p_2$ [dB]},
      ylabel={Mean square error $\EE[|\hat{\theta}_1-\theta_1|^2]$}
      ]
      \addplot[black,smooth,mark=*,line width=0.5pt] plot coordinates{
	      (-5,0.0430211673)(-2.500000e+00,0.0213757241)(0,0.0152115078)(2.500000e+00,0.0049219604)(5,0.0009135585)(7.500000e+00,0.0002234039)(10,0.0000796193)(1.250000e+01,0.0000337174)(15,0.0000165239)(1.750000e+01,0.0000083325)(20,0.0000045246)(2.250000e+01,0.0000024834)(25,0.0000013673)(2.750000e+01,0.0000007791)(30,0.0000004284)	     
      };
      \addplot[black,smooth,densely dashed,mark=*,line width=0.5pt] plot coordinates{
(-5,0.0238763990)(-2.500000e+00,0.0129273097)(0,0.0120808803)(2.500000e+00,0.0094201813)(5,0.0021727982)(7.500000e+00,0.0001627251)(10,0.0000542811)(1.250000e+01,0.0000273388)(15,0.0000146217)(1.750000e+01,0.0000078030)(20,0.0000043603)(2.250000e+01,0.0000024323)(25,0.0000013444)(2.750000e+01,0.0000007740)(30,0.0000004278)
      };
      \addplot[black,smooth,mark=o,line width=0.5pt] plot coordinates{
(-5,0.0595022176)(-2.500000e+00,0.0197784488)(0,0.0119315611)(2.500000e+00,0.0110344067)(5,0.0107941408)(7.500000e+00,0.0107267002)(10,0.0106980516)(1.250000e+01,0.0110224028)(15,0.0102868472)(1.750000e+01,0.0026175157)(20,0.0000530678)(2.250000e+01,0.0000081634)(25,0.0000034996)(2.750000e+01,0.0000016762)(30,0.0000008415)
      };
      \addplot[black,smooth,densely dashed,mark=o,line width=0.5pt] plot coordinates{
(-5,0.0444461715)(-2.500000e+00,0.0146243859)(0,0.0118111476)(2.500000e+00,0.0110121514)(5,0.0107768647)(7.500000e+00,0.0107145634)(10,0.0106859641)(1.250000e+01,0.0109943135)(15,0.0103874302)(1.750000e+01,0.0028065995)(20,0.0000571273)(2.250000e+01,0.0000086684)(25,0.0000035892)(2.750000e+01,0.0000017022)(30,0.0000008499)
      };
      \addplot[black,smooth,mark=triangle,line width=0.5pt] plot coordinates{
(-5,0.090728)(-2.500000e+00,0.079938)(0,0.038164)(2.500000e+00,0.010540)(5,0.009295)(7.500000e+00,0.009221)(10,0.009193)(1.250000e+01,0.009184)(15,0.009198)(1.750000e+01,0.009187)(20,0.009195)(2.250000e+01,0.009193)(25,0.009207)(2.750000e+01,0.009229)(30,0.009174)
      };
      \addplot[black,smooth,mark=square,line width=0.5pt] plot coordinates{
(-5,0.038562)(-2.500000e+00,0.014638)(0,0.011960)(2.500000e+00,0.011190)(5,0.011056)(7.500000e+00,0.011214)(10,0.013264)(1.250000e+01,0.029028)(15,0.018798)(1.750000e+01,0.009263)(20,0.009199)(2.250000e+01,0.009203)(25,0.009206)(2.750000e+01,0.009227)(30,0.009173)	      
      };
      \legend{ {Robust G-MUSIC}, {Emp. robust G-MUSIC}, {G-MUSIC},{Emp. G-MUSIC},{Robust MUSIC},{MUSIC} }
    \end{semilogyaxis}
  \end{tikzpicture}
  \caption{Means square error performance of the estimation of $\theta_1=10^\circ$, with $N=20$, $n=100$, two sources at $10^\circ$ and $12^\circ$, sample outlier scenario $\tau_i=1$, $i<n$, $\tau_n=100$, $u(x)=(1+\alpha)/(\alpha+x)$ with $\alpha=0.2$, $p_1=p_2$.}
  \label{fig:perf_Spike}
\end{figure}

\section{Concluding Remarks}
\label{sec:conclusion}
Robust estimators of scatter were originally designed to provide improved covariance (or scatter) matrix estimates of non-Gaussian zero mean random vectors, consistent in the regime $n\to\infty$, which are particularly suited to elliptical samples \citep{MAR76,TYL87} or to accommodate for outliers \citep{HUB64}. Similar to the more classical sample covariance matrix, the large $n$ consistency however falls short when the population size $N$ is large as well. Random matrix methods allows one to restore consistency in this regime by providing alternative estimation methods of spectral properties of the population covariance or scatter matrices. This is the result of a two-step method: (i) the analysis of the limiting spectrum of the covariance estimators (\citep{MAR67} for sample covariance matrices and \citep{COU13b} for robust estimates of scatter) and (ii) the introduction of improved statistical inference methods. For sample covariance matrices, Point (ii) is the result of the works of Girko \citep{GIR87} and more recently Mestre \citep{MES08b}. The present article provides a first instance of Point (ii) for robust estimators of scatter. The need here for a restriction to a spiked model (while \citep{GIR87,MES08b} treat more generic models) is intimately related to the structure of the approximation $\hat{S}_N$ of $\hat{C}_N$ which heavily depends on a non-observable variable $\gamma$ which may in general be itself an involved function of the parameters to be estimated.

The interest of robust methods is to harness the effect of rare sample outliers, the concatenation of which can be seen as a small rank perturbation matrix of the data sample matrix. A non obvious outcome of the present study is that, while sample covariance matrices equally treat small rank sample and population perturbations by creating non distinguishable spikes in the spectrum, robust estimates of scatter isolate sample versus population perturbations. This makes it possible to specifically estimate information carried by population perturbations, which is one important consequence of Theorem~\ref{th:2}. The practical purpose of this discriminative advantage is obvious and was exemplified by the introduction in Corollary~\ref{co:RGMUSIC} of an improved angle of arrival estimation method which is resilient to sample outliers.

However, since robust estimators of scatter are non unique (Maronna's estimators are defined through $u$ and other estimators such as Tyler's exist), this naturally raises the question of an optimal estimator choice. These questions demand more advanced studies on second order statistics for given performance metrics. Initial investigations are optimistic as they suggest that, on top of $\Vert \hat{C}_N-\hat{S}_N\Vert\asto 0$, differences of linear spectrum functionals of the type $\int f d\mu_{\hat{C}_N}-\int f d\mu_{\hat{S}_N}$, with $\mu_X$ the empirical spectral distribution of $X$ and $f$ a continuous and bounded function, have much weaker fluctuations than each integral around its mean; this indicates that fluctuations of functionals of $\hat{C}_N$ can be studied equivalently through the much more tractable fluctuations of functionals of $\hat{S}_N$.

\section{Proof of the main results}
\label{sec:proof}
\subsection{Notations}

Throughout the proof, we shall use the following shortcut notations:
\begin{align*}
	T &= \diag(\{\tau_i\}_{i=1}^n) \in \CC^{n\times n}\\
	V &= \diag(\{v_c(\tau_i\gamma\}_{i=1}^n) \in \CC^{n\times n}\\
	S &= \left[\{s_{ij}\}_{1\leq i\leq L,1\leq j\leq n}\right] \in \CC^{L\times n} \\
	W &= [w_1,\ldots,w_n] \in\CC^{N\times n} \\
	\tilde{W} &= [\tilde{w}_1,\ldots,\tilde{w}_n] \in\CC^{N\times n}
\end{align*}
with $\tilde{w}_i=w_ir_i/\sqrt{N}$ as in the statement of Theorem~\ref{th:1}. We shall expand $A$ as the singular value decomposition $A=U\Omega \bar{U}^*$ with $U\in\CC^{N\times L}$ isometric, $\Omega=\diag(\sigma_1,\ldots,\sigma_L)$, $\sigma_1\geq \ldots\geq \lambda_L\geq 0$, and $\bar{U}\in\CC^{L\times L}$ unitary. 

We also define
\begin{align*}
	\hat{S}_N^\circ  &= \frac1n\sum_{i=1}^n \tau_i v_c(\tau_i\gamma) \tilde{w}_i\tilde{w}_i^* = \frac1n \tilde{W}TV\tilde{W}^*
\end{align*}
which corresponds to $\hat{S}_N$ with $p_1=\ldots=p_L=0$, i.e., with no perturbation, and 
\begin{align*}
	Q_z^\circ&=(\hat{S}_N^\circ -zI_N)^{-1}=\left(\frac1n\tilde{W}TV\tilde{W}^*-zI_N\right)^{-1}
\end{align*}	
the resolvent of $\hat{S}_N^\circ$. 

For couples $(\eta,M_\eta)$, $\eta<1$, such that $\nu( (0,M_\eta) )>0$ and $\nu( (M_\eta,\infty) )<\eta$, it will be necessary to define $T_\eta$ the matrix $T$ in which all values of $\tau_i$ greater or equal to $M_\eta$ are replaced by zeros, and similarly for $V_\eta$. Denote also $\gamma^\eta$ the unique solution to
	\begin{align}
		\label{eq:gamma_eta}
		1 &= \int_{\tau<M_\eta} \frac{\psi_c(\tau\gamma^\eta)\nu(d\tau)}{1+c\psi_c(\tau\gamma^\eta)}.
	\end{align}
	and $\hat{S}_{N,\eta}$ the resulting $\hat{S}_N$ matrix with all $\tau_i$ greater than $M_\eta$ discarded and $\gamma$ replaced by $\gamma^\eta$. 

	Finally, we further define $T_{(j)}=\diag(\{\tau_i\}_{i\neq j})$ and similarly for $V_{(j)}$, $S_{(j)}$, $\tilde{W}_{(j)}$, $\hat{S}_{(j)}=\hat{S}_{N,(j)}$  the matrices with column or component $j$ discarded, as well as $T_{(j),\eta}$ the matrix $T_\eta$ with row-and-column $j$ discarded, and similarly $V_{(j),\eta}$, $S_{(j),\eta}$, $\tilde{W}_{(j),\eta}$, $\hat{S}_{(j),\eta}$ the corresponding matrices with column or component $j$ discarded.

\subsection{Overall proof strategy}

The existence and uniqueness of $\hat{C}_N$ as defined in the statement of Theorem~\ref{th:1} follows immediately from the recent work \citep{CHI14} (which is more general than the previous result \cite[Theorem~1]{COU13b}). One of the key elements of the proof of convergence in Theorem~\ref{th:1} is to ensure that there exists $\varepsilon>0$ such that, for all large $n$ a.s., all eigenvalues of $\{\hat{S}_{(j)},1\leq j\leq n\}$ (and also of $\{\hat{S}_{(j),\eta},1\leq j\leq n\}$ for given $\eta$ small) are greater than $\varepsilon$. This is an important condition to ensure that the quadratic forms $\frac1N\tilde{w}_j^*\hat{S}_{(j)}^{-1}\tilde{w}_j$, which play a central role in the proof, are jointly controllable. In \citep{COU13b}, where the convergence $\Vert \hat{C}_N-\hat{S}_N\Vert\asto 0$ is obtained for $p_1=\ldots=p_L=0$, this unfolded readily from \cite[Lemma~2]{COU13} (i.e., \cite[Lemma~2]{COU13} states that the matrices $\frac1n\tilde{W}_{(j)}\tilde{W}_{(j)}^*$ have their smallest eigenvalue uniformly away from zero). Here, due to the existence of a small rank matrix $A$, the approach from \cite[Lemma~2]{COU13} no longer holds as $\hat{S}_{(j)}$ may a priori exhibit finitely many isolated eigenvalues getting close to zero as $n\to\infty$. We shall show that this is not possible. Precisely, we shall prove that the large $n$ spectrum of $\hat{S}_N$ is similar to that of $\hat{S}_N^\circ$ but possibly for finitely many isolated eigenvalues, none of which can be asymptotically found close to zero. We shall however characterize those eigenvalues of $\hat{S}_N$ found beyond the right-edge of the limiting spectrum of $\hat{S}_N^\circ$. Once this result is obtained, to complete the proof of Theorem~\ref{th:1}, it will then suffice to check that most spectral statistics involved in the proof of \cite[Theorem~2]{COU13b} are not affected by the presence of the additional small rank matrix $AS$ in the model. Since most results need be proved jointly for the matrix sets $\{\hat{S}_{(j)},1\leq j\leq n\}$ (or $\{\hat{S}_{(j),\eta},1\leq j\leq n\}$), high order moment bounds will be required to then apply union bound along with Markov inequality techniques. As the proof in \citep{COU13b} is rather long and technical and since the main contribution of the present article lies in Theorem~\ref{th:2}, we only discuss in what follows the main new technical elements that differ from \citep{COU13b}.

When Theorem~\ref{th:1} is obtained, the proofs of Theorems~\ref{th:2} and \ref{th:3} unfolds from classical techniques for spiked random matrix models, using the approximation $\hat{S}_N$ for $\hat{C}_N$. The model $\hat{S}_N$ considered here is closely related to the scenario of \citep{CHA12}, but for the random non-Gaussian structure of the matrix $S$; also, \citep{CHA12} imposes $\max_i{\rm dist}(\tau_i,{\rm Supp}(\nu))\to 0$ which we do not enforce here.

\subsection{Localization of the eigenvalues of $\hat{S}_N$ and $\hat{S}_{(i)}$}
\label{sec:localization}

We first study the localization of the eigenvalues of $\hat{S}_N$ and $\{\hat{S}_{(j),\eta},1\leq j\leq n\}$. The strategy being the same, we concentrate mostly on the study of $\hat{S}_N$ and then briefly generalize the approach to $\{\hat{S}_{(j),\eta},1\leq j\leq n\}$.

By isolating the small rank perturbation terms, we first develop $\hat{S}_N$ as
\begin{align*}
	\hat{S}_N &= \frac1n\sum_{i=1}^n v_c(\tau_i\gamma) A_i\bar{w}_i\bar{w}_i^*A_i^* \\
	&= \hat{S}_N^\circ + \frac1n ASVS^*A^* + \frac1n AST^{\frac12}V\tilde{W}^* + \frac1n \tilde{W}T^{\frac12}VS^*A^*.
\end{align*}

Let $\lambda\in \RR\setminus [\varepsilon,S^++\varepsilon]$ for some $\varepsilon>0$ small be an eigenvalue of $\hat{S}_N$. Note that such a $\lambda$ may not exist. However, from \citep{SIL98} and since in particular $\limsup_n \Vert AA^*\Vert <\infty$ and $\limsup_n \Vert T^{\frac12}V\Vert <\infty$, the spectral norm of each matrix above is asymptotically bounded almost surely and thus $\limsup_n \lambda<\infty$ a.s. Also, from \citep{COU13b} and from the discussion prior to the statement of Theorem~\ref{th:1}, for all large $n$ a.s., $\lambda$ is not an eigenvalue of $\hat{S}_N^\circ$ (for $\varepsilon$ chosen small enough). Thus, by definition, $\lambda$ is a solution of $\det(\hat{S}_N-\lambda I_N)=0$ while $\Vert (\hat{S}_N^\circ-\lambda I_N)^{-1}\Vert < M$ for some $M>0$ independent of $n$ but increasing as $\varepsilon\to 0$. As such, from the development above, for all large $n$ a.s.,
\begin{align*}
	0 &= \det\left( \hat{S}_N^\circ - \lambda I_N + \Gamma \right) = \det \left( Q_\lambda^\circ \right)^{-1} \det \left( I_N + (Q_\lambda^\circ)^\frac12 \Gamma (Q_\lambda^\circ)^\frac12 \right)
\end{align*}
where $\Gamma=\frac1n ASVS^*A^* + \frac1n AST^{\frac12}V\tilde{W}^* + \frac1n \tilde{W}T^{\frac12}VS^*A^*$ can be further written
\begin{align}
	\label{eq:Gamma}
	\Gamma &= \begin{bmatrix} U\Omega^{\frac12} & \frac1n\tilde{W}T^{\frac12}VS^*\bar{U}\Omega^{\frac12} \end{bmatrix}\begin{bmatrix} \Omega^{\frac12}\bar{U}^*\frac1n\tilde{W}V\tilde{W}^*\bar{U}\Omega^{\frac12} & I_L \\ I_L & 0 \end{bmatrix} \begin{bmatrix} \Omega^{\frac12}U^* \\ \Omega^{\frac12}\frac1n\bar{U}^*ST^{\frac12}V\tilde{W} \end{bmatrix}.
\end{align}
Exploiting the small rank of $S$ and $A$, and the formula $\det(I+AB)=\det(I+BA)$ for properly sized $A,B$ matrices, this induces
\begin{align*}
	0 &=\det \left( I_{2L} + \Gamma_L(\lambda) \right)
\end{align*}
where
\begin{align*}
	\Gamma_L(\lambda) &\triangleq \begin{bmatrix} \Omega^{\frac12}\bar{U}^*\frac1n\tilde{W}V\tilde{W}^*\bar{U}\Omega^{\frac12} & I_L \\ I_L & 0  \end{bmatrix} \begin{bmatrix} \Omega^{\frac12}U^* \\ \Omega^{\frac12}\frac1n\bar{U}^*ST^{\frac12}V\tilde{W} \end{bmatrix} Q_\lambda^\circ \begin{bmatrix} U\Omega^{\frac12} & \frac1n\tilde{W}T^{\frac12}VS^*\bar{U}\Omega^{\frac12} \end{bmatrix}.
\end{align*}

We now need the following central lemmas.
	\begin{lemma}
		\label{le:1}
		Let $\varepsilon>0$ and $\mathcal A_\varepsilon$ be the event $\varepsilon<\lambda_N(\hat{S}_N^\circ)<\lambda_1(\hat{S}_N^\circ)< S^++\varepsilon$. Let also $a,b\in\CC^N$ be two vectors of unit norm. Then, for every $z\in\mathcal C \subset \CC\setminus [\varepsilon,S^++\varepsilon]$ with $\mathcal C$ compact,
		\begin{align*}
			\EE\left[ \left| \frac1nS^*VS - \frac1n\tr V \right|^p \right] &\leq KN^{-\frac{p}2} \\
			\EE\left[1_{\mathcal A_\varepsilon} \left| \frac1n ST^{\frac12}V\frac1n\tilde{W}^*Q_z^\circ \tilde{W} VT^{\frac12}S^* - \left[\frac1n\tr V + z \frac1n\tr V \tilde{Q}_z^\circ \right]\right|^p \right] &\leq KN^{-\frac{p}2} \\
			\EE\left[1_{\mathcal A_\varepsilon} \left|  a^*Q_z^\circ b - a^*b\frac1N\tr Q_z^\circ \right|^p \right] &\leq KN^{-\frac{p}2} \\
			\EE\left[1_{\mathcal A_\varepsilon} \left\Vert \frac1n a^*Q_z^\circ \tilde{W}T^{\frac12}V S^*\right\Vert^p \right] &\leq KN^{-\frac{p}2}
		\end{align*}
		where $\tilde{Q}_z^\circ=(\frac1nT^{\frac12}V^{\frac12}\tilde{W}^*\tilde{W}V^{\frac12}T^{\frac12}-zI_N)^{-1}$ and $K>0$ does not depend on $z$.
	\end{lemma}
	\begin{proof}
		The first convergence is a mere application of \cite[Lemma~B.26]{SIL06}. Similarly, noticing that 
		\begin{align*}
			\frac1n ST^{\frac12}V\frac1n\tilde{W}^*Q_z^\circ \tilde{W} VT^{\frac12}S^* &= \frac1n SV^{\frac12} \left[T^{\frac12}V^{\frac12}\frac1n\tilde{W}^*\tilde{W} V^{\frac12}T^{\frac12}\tilde{Q}_z^\circ\right] V^{\frac12} S^* \\
			&= \frac1n SVS^* + z \frac1n S\tilde{Q}_z^\circ V^{\frac12} S^* 
		\end{align*}
		the second result follows again by \cite[Lemma~B.26]{SIL06} and the fact that $\limsup_n \Vert\tilde{Q}_z^\circ\Vert<1/{\rm dist}(\mathcal C,[\varepsilon,S^++\varepsilon])$. Using the fact that $\tilde{W}$ is Gaussian, the third result follows from the same proof as in \cite[Lemma~3]{LOU10b} using additionally $[VT]_{ii}<\psi_\infty$. Similarly, conditioning first on $S$, which is independent of $\tilde{W}$, we obtain by the same proof as in \cite[Lemma~4]{LOU10b} that
		\begin{align*}
			\EE_{\tilde{W}}\left[1_{\mathcal A_\varepsilon} \left| \frac1n a^*Q_z^\circ \tilde{W}T^{\frac12}V s_i\right|^p \right] &\leq K \Vert n^{-\frac12}s_i\Vert^p N^{-\frac{p}2}
		\end{align*}
		where we denoted $S^*=[s_1,\ldots,s_L]$ (the proof follows from exploiting the left-unitary invariance of $\tilde{W}$ and applying the integration by parts and Poincar\'e--Nash inequality method for unitary Haar matrices described in \cite[Chapter~8]{PAS11}). Now, $\EE[\Vert n^{-\frac12}s_i\Vert^p]=O(1)$ by H\"older's inequality, and we obtain the last inequality. 
	\end{proof}

	\begin{lemma}
		\label{le:2}
		For $z\in \CC\setminus [S_\mu^-,S_\mu^+]$, let $\delta(z)$ be the unique solution to the equation
		\begin{align*}
			\delta(z) &= c\left( -z + \int \frac{t v_c(t \gamma)}{1+\delta(z)t v_c(t \gamma)} d\nu(t) \right)^{-1}
		\end{align*}
		where we recall that $\gamma$ is the unique positive solution to
		\begin{align*}
			1 &= \int \frac{\psi_c(t \gamma)}{1+c\psi_c(t\gamma)}d\nu(t).
		\end{align*}
		Let now $z\in \mathcal C$, with $\mathcal C$ a compact set of $\CC\setminus [\varepsilon,S_\mu^++\varepsilon]$ for some $\varepsilon>0$ small enough. Then, denoting $\Psi_z^\circ=(I_n + \delta(z)VT)^{-1}$,
		\begin{align*}
			\sup_{z\in\mathcal C}\left| \frac1N\tr Q_z^\circ - \frac{\delta(z)}c \right| &\asto 0 \\
			\sup_{z\in\mathcal C}\left| \frac1n\tr V+z\frac1n\tr V \tilde{Q}_z^\circ - \delta(z)\frac1n\tr V^2T\Psi_z^\circ\right| &\asto 0.
		\end{align*}
	\end{lemma}
	\begin{proof}
		The almost sure convergences to zero of the terms inside the norms (i.e., for each $z\in\mathcal C$) are classical, see e.g., \citep{SIL95}. Considering a countable sequence $z_1,z_2,\ldots$ of such $z\in\mathcal C$ having an accumulation point, by the union bound, there exists a probability one set on which the convergence is valid for each point of the sequence. Now, by \citep{COU13b}, for all large $n$ a.s., $Q_z^\circ$ and $\tilde Q_z^\circ$ are analytic on $\mathcal C$. Since $\delta(z)$ is also analytic on $\mathcal C$, by Vitali's convergence theorem \citep{TIT39}, the convergences are uniform on $\mathcal C$.
	\end{proof}

	From \citep{COU13b} again, for $\varepsilon>0$ small enough, the set $\mathcal A_\varepsilon$ introduced in Lemma~\ref{le:1} satisfies $1_{\mathcal A_\varepsilon}\asto 1$. As such, using the Markov inequality and the Borel Cantelli lemma, Lemma~\ref{le:1} for $p>2$ ensures that all quantities in absolute values in the statement of Lemma~\ref{le:1} converge to zero almost surely as $n\to\infty$. Since the quantities involved are analytic on compact $\mathcal C\subset \CC \setminus [\varepsilon,S^++\varepsilon]$, considering a countable sequence of $z\in\mathcal C$ having a limit point, it is clear by Vitali's convergence theorem \citep{TIT39} that these convergences are uniform on $\mathcal C$. Applying successively Lemma~\ref{le:1} for $p>2$ and Lemma~\ref{le:2}, we then obtain, for $\mathcal C\subset \CC\setminus [\varepsilon,S^++\varepsilon]$,
	\begin{align*}
		\sup_{z\in\mathcal C}	\left\{ \left\Vert \Gamma_L(z) - \begin{bmatrix} \Omega \frac1n\tr V & I_L \\ I_L & 0 \end{bmatrix} \begin{bmatrix} \Omega \frac{\delta(z)}c & 0 \\ 0 & \Omega \delta(z) \frac1n\tr V^2T\Psi_z^\circ \end{bmatrix} \right\Vert \right\} \asto 0
	\end{align*}
	or equivalently
	\begin{align}
		\label{eq:cvg_GammaL}
		\sup_{z\in\mathcal C} \left\{\left\Vert \Gamma_L(z) - \begin{bmatrix} \Omega^2 \frac{\delta(z)}c \frac1n\tr V & \Omega \delta(z) \frac1n\tr V^2T\Psi_z^\circ \\ \Omega \frac{\delta(z)}c & 0 \end{bmatrix} \right\Vert\right\} \asto 0.
	\end{align}
	We may then particularize this result to $z=\lambda$ which, for $\varepsilon$ sufficiently small, remains bounded away from $[\varepsilon,S^++\varepsilon]$ as $n$ grows (but of course depends on $n$) to obtain
	\begin{align}
		\label{eq:cvg_GammaL_lambda}
		\left\Vert \Gamma_L(\lambda) - \begin{bmatrix} \Omega^2 \frac{\delta(\lambda)}c \frac1n\tr V & \Omega \delta(\lambda) \frac1n\tr V^2T\Psi_\lambda^\circ \\ \Omega \frac{\delta(\lambda)}c & 0 \end{bmatrix} \right\Vert \asto 0.
	\end{align}

	For $\bar{\lambda}\in \RR \setminus [\varepsilon,S^++\varepsilon]$, let us now study the equation
	\begin{align}
		\label{eq:detlimit}
		\det \left( I_{2L} + \begin{bmatrix} \Omega^2 \frac{\delta(\bar{\lambda})}c \frac1n\tr V & \Omega \delta(\bar{\lambda}) \frac1n\tr V^2T\Psi_{\bar{\lambda}}^\circ \\ \Omega \frac{\delta(\bar{\lambda})}c & 0 \end{bmatrix} \right) = 0.
	\end{align}
	After development of the determinant, this equation is equivalent to 
	\begin{align*}
		\sigma^2_{\ell}\frac{\delta(\bar{\lambda})}c\left( \frac1n\tr V - \delta(\bar{\lambda}) \frac1n\tr V^2T\Psi_{\bar{\lambda}}^\circ \right) + 1 &= 0
	\end{align*}
	for some $\ell\in\{1,\ldots,L\}$, or equivalently, using $V-\delta(\bar{\lambda})V^2T\Psi_{\bar\lambda}^\circ=V\Psi_{\bar\lambda}^\circ$
	\begin{align*}
		\sigma^2_{\ell} \delta(\bar{\lambda}) \frac1N\sum_{i=1}^n \frac{v_c(\tau_i\gamma)}{1+\tau_i v_c(\tau_i\gamma)\delta(\bar{\lambda})} &= 0.
	\end{align*}
	In the limit $n\to\infty$, using $A^*A\asto \diag(p_1,\ldots,p_L)$ and $\frac1n\sum_{i=1}^n{\bm\delta}_{\tau_i}\to \nu$ a.s., any accumulation point $\bar{\Lambda}\in (\RR \setminus (\varepsilon,S^++\varepsilon))\cup \{\infty\}$ of $\bar{\lambda}$ must satisfy
	\begin{align}
		\label{eq:eq_p}
		1 + p_{\ell} \frac1c \int \frac{ \delta(\bar{\Lambda}) v_c(\tau\gamma)}{1+\delta(\bar{\Lambda})\tau v_c(\tau\gamma)}\nu(dt) = 0.
	\end{align} 
This unfolds from dominated convergence, using $\delta( (S^+,\infty) )\subset (-(\tau_+v(\tau_+\gamma))^{-1},0)$ with $\tau_+\in(0,\infty]$ the right-edge of the support of $\nu$; in particular, if ${\rm Supp}(\nu)$ is unbounded, $\delta( (S^+,\infty) )\subset (-\gamma/\psi_\infty,0)$ \citep{HAC13}. 
	Let us then consider the equation in the variable $\Lambda\in (S^+,\infty)$
	\begin{align}
		\label{eq:Lambda}
		-\left(\frac1c\int \frac{\delta(\Lambda) v_c(\tau \gamma)}{1+\delta(\Lambda)\tau v_c(\tau\gamma)} \nu(d\tau)\right)^{-1}= p_{\ell}.
	\end{align}

	We know from \citep{COU13b} that, since $\nu([0,m))<1-\phi_\infty^{-1}$ for some $m>0$ (by Assumption~\ref{ass:u}), $S_\mu^->0$. Also, as the Stieltjes transform of a measure with support included in $[S_\mu^-,S_{\mu}^+]\subset [S_\mu^-,S^+]$, $\delta$ is increasing on both $[0,S_\mu^-)$ and $(S^+,\infty)$. Moreover, $\delta( [0,S_\mu^-) )\subset (0,\infty)$ and $\delta( (S^+,\infty) )\subset (-(\tau_+v(\tau_+\gamma))^{-1},0)$. 
	Therefore, the left-hand side of \eqref{eq:Lambda} is negative for $\Lambda\in [0,S_\mu^-)$ and the equation has no solution in this set. It is now easily seen that the left-hand side of \eqref{eq:Lambda} is increasing with $\Lambda$ with limits infinity as $\Lambda\to\infty$ and $p_->0$ as $\Lambda\downarrow S^+$. Therefore, if $p_-<p_{\ell}$, the above equation has a unique solution $\Lambda_{\ell}\in (S^+,\infty)$, distinct for each distinct $p_{\ell}$. Hence, $\bar{\lambda}\to \bar{\Lambda}=\Lambda_{\ell}$. 
	
		By the argument principal, for all $n$ large a.s., the number of eigenvalues of $\hat{S}_N$, i.e., the number of zeros of $\det(I_{2L}+\Gamma_L(\lambda))$, in any open set $\mathcal V\subset \RR\setminus [\varepsilon,S^++\varepsilon]$ is
	\begin{align*}
		\frac1{2\pi\imath} \oint_{\mathcal I} \frac{[\det(I_{2L}+\Gamma_L(z))]'}{\det(I_{2L}+\Gamma_L(z))}dz
	\end{align*}
	with $\mathcal I$ a contour enclosing $\mathcal V$. By the uniform convergence of \eqref{eq:cvg_GammaL} on $\mathcal V$, the analyticity of the quantities involved, and the fact that the involved determinant is a polynomial of order at most $2L$ of its entries, this value asymptotically corresponds to the number of solutions to \eqref{eq:detlimit} in $\mathcal V$ counted with multiplicity, which in the limit are the $\Omega_k\in \mathcal V$. Particularizing $\mathcal V$ to $(-1,2\varepsilon)$ for $\varepsilon>0$ small enough and then to any small open ball around $\Lambda_\ell$ for each $\ell$ such that $p_\ell>p_-$, we then conclude that $\hat{S}_N$ has asymptotically no eigenvalue in $[0,\varepsilon]$ but that $\lambda_{\ell}(\hat{S}_N)\asto \Lambda_{\ell}$ for all $\ell\in\mathcal L$, which is the expected result.

\bigskip

The precise localization of the eigenvalues of $\hat{S}_N$ will be fundamental for the proof of Theorems~\ref{th:2}~and~\ref{th:3}. To prove Theorem~\ref{th:1} though, we need to generalize part of this result to the matrices $\hat{S}_{(j)}$ and $\hat{S}_{(j),\eta}$ defined at the beginning of the section. Precisely, we need to show that there exists $\varepsilon>0$ such that $\min_{1\leq j\leq N}\{\lambda_N(\hat{S}_{(j)})\}>\varepsilon$ for all large $n$ a.s., and similarly for $\hat{S}_{(j),\eta}$.

Take $j\in\{1,\ldots,n\}$. Replacing $\hat{S}_N$ by $\hat{S}_{(j)}$ in the proof above leads to the same conclusions. Indeed, by a rank-one perturbation argument \cite[Lemma~2.6]{SIL95}, for each $\varepsilon>0$, for all large $n$ a.s.
\begin{align*}
	\frac1n\tr \tilde{Q}_z^\circ- \frac1n\tr \left(\frac1n\tilde{W}_{(j)}T_{(j)}V_{(j)}\tilde{W}_{(j)}-zI_N\right)^{-1}\leq \frac1n \frac1{ {\rm dist}(z,[\varepsilon,S^++\varepsilon])}
\end{align*}
and therefore, up to replacing all matrices $X$ by $X_{(j)}$ in their statements, Lemmas~\ref{le:1}~and~\ref{le:2} hold identically (with $\delta(z)$ unchanged). Exploiting $\frac1{n-1}\sum_{i\neq j}{\bm\delta}_{\tau_i}\to \nu$ a.s., the remainder of the proof unfolds all the same and we have in particular that for all large $n$ a.s. $\hat{S}_{(j)}$ has no eigenvalue below some $\varepsilon>0$. 

We now prove that this result can be made uniform across $j$. Denote $\Gamma_{L,(j)}(z)$ the matrix $\Gamma_L(z)$ with all matrices $X$ replaced by $X_{(j)}$. Also rename Lemmas~\ref{le:1}~and~\ref{le:2} respectively Lemma~\ref{le:1}-$(j)$ and Lemma~\ref{le:2}-$(j)$, and rename $\mathcal A_\varepsilon$ by $\mathcal A_{\varepsilon,(j)}$ in the statement of Lemma~\ref{le:1}-$(j)$. Then, taking $p>4$ in Lemma~\ref{le:1}-$(j)$, by the union bound and the Markov inequality, for $e>0$,
\begin{align*}
	&P\left( \max_{1\leq j\leq n} 1_{A_{\varepsilon,(j)}} \left\Vert \Gamma_{L,(j)}(z) - \begin{bmatrix} \Omega^2 \frac{\delta(z)}c \frac1n\tr V & \Omega \delta(z) \frac1n\tr V^2T\Psi_z^\circ \\ \Omega \frac{\delta(z)}c & 0 \end{bmatrix} \right\Vert > e  \right) \nonumber \\
	&\leq  \frac1{e^p}\sum_{j=1}^n \EE\left[ 1_{A_{\varepsilon,(j)}} \left\Vert \Gamma_{L,(j)}(z) - \begin{bmatrix} \Omega^2 \frac{\delta(z)}c \frac1n\tr V & \Omega \delta(z) \frac1n\tr V^2T\Psi_z^\circ \\ \Omega \frac{\delta(z)}c & 0 \end{bmatrix} \right\Vert^p \right] \nonumber \\
	&= O(N^{ 1-\frac{p}2 })
\end{align*}
which is summable. By the Borel Cantelli lemma, the event in the probability parentheses then converges a.s.\@ to zero. Finally, from \citep{COU13}, there exists $\varepsilon>0$ such that $1_{\cap_{j=1}^n A_{\varepsilon,(j)}}\asto 1$. We then conclude that, for each $z\in\mathcal C\subset \CC\setminus [\varepsilon,S_\mu^++\varepsilon]$ for some $\varepsilon>0$,
\begin{align*}
	\sup_{1\leq j\leq n} \left\Vert \Gamma_{L,(j)}(z) - \begin{bmatrix} \Omega^2 \frac{\delta(z)}c \frac1n\tr V & \Omega \delta(z) \frac1n\tr V^2T\Psi_z^\circ \\ \Omega \frac{\delta(z)}c & 0 \end{bmatrix} \right\Vert \asto 0
\end{align*}
Let now $\mathcal V\subset \CC\setminus [\varepsilon,S_\mu^++\varepsilon]$ be a bounded open set containing $[0,\varepsilon/2]$ and $\mathcal I$ be its smooth boundary. Taking the determinant of each matrix inside the norm and using again the analyticity of the functions involved, we now get that the quantity
\begin{align*}
	\frac1{2\pi\imath} \oint_{\mathcal I} \frac{[\det(I_{2L}+\Gamma_{L,(j)}(z))]'}{\det(I_{2L}+\Gamma_{L,(j)}(z))}dz
\end{align*}
converges almost surely uniformly across $j\in\{1,\ldots,n\}$ to the number of eigenvalues of any of the $\hat{S}_{(j)}$ within $[0,\varepsilon/2]$. But by the previous proof, this must be zero. Hence, for all large $n$ a.s., none of the $\hat{S}_{N,(j)}$ has eigenvalues smaller than $\varepsilon/2$, which is what we wanted.

\bigskip

Let now $(\eta,M_\eta)$ be such that $\nu( (0,M_\eta) )>0$ and $\nu( (M_\eta,\infty) )<\eta$. We have now $\frac1n\sum_{i=1}^n 1_{\tau_i\leq M\eta} {\bm \delta}_{\tau_i}\asto \nu\eta\triangleq c_\eta \nu + (1-c_\eta){\bm\delta}_0$ with $c_\eta=\lim_n n^{-1}|\{\tau_i\leq M_\eta\}|=1-\eta$ (which almost surely exists by the law of large numbers), so that $\nu_\eta( [0,m) )<\eta+(1-\eta)(1-\phi_\infty^{-1})$ for some $m>0$ (Assumption~\ref{ass:u}). Taking $\eta$ small enough so that $\nu_\eta( [0,m) )<1-\phi_\infty^{-1}$, we are still under the assumptions of \cite[Theorem~2]{COU13b} and therefore we again have that for all large $n$ a.s.\@ none of the matrices $\hat{S}_{(j),\eta}$ has eigenvalues below a certain positive value $\varepsilon_\eta>0$.

These elements are sufficient to now turn to the proof of the main theorems.

\subsection{Proof of Theorem~\ref{th:1}}

When $p_1=\ldots=p_L=0$, Theorem~\ref{th:1} unfolds directly from \cite[Theorem~2]{COU13b}. Indeed, in this scenario, the latter result states
\begin{align}
	\label{eq:res_COU13b}
	\left\Vert \hat{C}_N - \frac1n\sum_{i=1}^n v(\tau_i \gamma_N) w_iw_i^* \right\Vert &\asto 0
\end{align}
with $\gamma_N$ the unique positive solution to
\begin{align*}
	1 &= \frac1n\sum_{i=1}^n \frac{\psi(\tau_i\gamma_N)}{1+c_n \psi(\tau_i\gamma_N)}.
\end{align*}
Using $\frac1n\sum_{i=1}^n {\bm\delta}_{\tau_i}\asto \nu$, $c_n\to c$, along with the boundedness of $\psi$, we have that any accumulation point $\gamma\in[0,\infty]$ of $\gamma_N$ as $n\to\infty$ must satisfy 
\begin{align*}
	1 &= \int \frac{\psi_c(\tau\gamma)\nu(d\tau)}{1+c\psi_c(\tau\gamma)}
\end{align*}
the solution of which is easily shown to be unique in $(0,\infty)$ as the right-hand side term is increasing in $\gamma$ with limits zero as $\gamma\to 0$ and $\psi_\infty>1$ as $\gamma\to\infty$ (unless $\nu={\bm\delta}_0$ which is excluded). Using the continuity and boundedness of $v$, it then comes $\max_i |v(\tau_i\gamma_N)-v_c(\tau_i\gamma)|\asto 0$.
Now, $w_iw_i^*=(w_iw_i^*r_i^2/N)/(r_i^2/N)$ where in the numerator $w_ir_i/\sqrt{N}$ is Gaussian and where the denominator satisfies $\max_i|r_i^2/N-1|\asto 0$ (using classical probability bounds on the chi-square distribution). With these results, along with \citep{SIL98} which ensures that $\frac1{nN}\sum_i w_iw_i^*r_i^2$ has bounded spectral norm for all large $n$ a.s., \cite[Theorem~2]{COU13b} implies
\begin{align*}
	\left\Vert \hat{C}_N - \frac1{nN}\sum_{i=1}^n v_c(\tau_i \gamma) w_iw_i^*r_i^2 \right\Vert &\asto 0
\end{align*}
which is the desired result for $p_1=\ldots=p_L=0$.

The generalization to generic $p_1,\ldots,p_L$ follows from a careful control of the elements of proof of \cite[Theorem~2]{COU13b}. We see that \cite[Lemma~1]{COU13b} and \cite[Remark~1]{COU13b} are not affected by $p_1,\ldots,p_L$ as these results only depend on $\tau_1,\ldots,\tau_n$. The fundamental lemma \cite[Lemma~2]{COU13b} (and its extension remark \cite[Remark~2]{COU13b}) as well as the lemma \cite[Lemma~3]{COU13b} however need be updated. 

We shall not go into the details of every generalization which is painstaking and in fact similar for each lemma. Instead, we detail the generalization of the important remark \cite[Remark~2]{COU13b} and merely give elements for the other results. The remark \cite[Remark~2]{COU13b} is now updated as follows.
\begin{lemma}
	\label{le:remark2COU13b}
	Let $(\eta,M_\eta)$ be couples indexed by $\eta\in(0,1)$ such that $\nu( (0,M_\eta) ) >0$ and $\nu( (M_\eta,\infty) )<\eta$ and define $\gamma^\eta$ as the unique solution to \eqref{eq:gamma_eta}. Also let $M>0$ be arbitrary. Then, for all $\eta$ small enough,
	\begin{align}
		\max_{\substack{1\leq j\leq n \\ \tau_j\leq M}} \left| \frac1Ny_j^* \left( \frac1n\sum_{\tau_i\leq M_\eta,i\neq j} v\left( \tau_i\gamma^\eta\right) y_iy_i^* \right)^{-1}y_j - \tau_j\gamma^\eta \right| &\asto 0 \nonumber \\
		\label{eq:lemma_gammaeta}
		\max_{\substack{1\leq j\leq n \\ \tau_j > M}} \left| \frac1{\tau_j}\frac1Ny_j^* \left( \frac1n\sum_{\tau_i\leq M_\eta,i\neq j} v\left( \tau_i\gamma^\eta\right) y_iy_i^* \right)^{-1}y_j - \gamma^\eta \right| &\asto 0.
	\end{align}
\end{lemma}
\begin{proof}
	Note that, replacing the terms $y_i$ by $\tau_i w_i$ in \eqref{eq:lemma_gammaeta} gives exactly \cite[Remark~2]{COU13b}.
	To ensure that the result holds, we then only need verify that the terms involving $AS$ become negligible. 

	For $\eta$ sufficiently small, define 
	\begin{align*}
		\check{S}_{(j),\eta} &=\frac1n\sum_{\tau_i\leq M_\eta,i\neq j} v\left( \tau_i\gamma^\eta\right) y_iy_i^* = \frac1n (AS_{(j)}+W_{(j)})V_{(j),\eta}(AS_{(j)}+W_{(j)})^*.
	\end{align*}
	Using the fact that $\max_{1\leq i\leq n}\{|r_i/\sqrt{N}-1|\}\asto 0$ and that all matrices in the equality above have bounded norm almost surely by \citep{SIL98}, we then have $\sup_{1\leq j\leq n}\Vert \check{S}_{(j),\eta} - \hat{S}_{(j),\eta}\Vert \asto 0$. From the results in the previous section, we then conclude that there exists $\varepsilon>0$ such that the eigenvalues of $\check{S}_{(j),\eta}$ for all $j$ are all greater than $\varepsilon$ for all large $n$ almost surely. Now, recalling that $S=[s_1,\ldots,s_n]$,
	\begin{align*}
		\frac1Ny_j^* \check{S}_{(j),\eta}^{-1}y_j &= \frac1Ns_j^*A^*\check{S}_{(j),\eta}^{-1}As_j + 2\Re \left[\sqrt{\tau_j}\frac1Ns_j^*A^*\check{S}_{(j),\eta}^{-1}w_j\right] + \tau_j\frac1N w_j^*\check{S}_{(j),\eta}^{-1}w_j.
	\end{align*}
	By the trace lemma \cite[Lemma~B.26]{SIL06}, denoting $\mathcal A$ the probability set over which the eigenvalues of $\check{S}_{(j),\eta}$ for all $j$ are greater than $\varepsilon$, for each $p>2$,
	\begin{align*}
		\EE\left[1_{\mathcal A} \left| \frac1N \tilde{w}_j^*\check{S}_{(j),\eta}^{-1}\tilde{w}_j - \frac1N\tr \check{S}_{(j),\eta}^{-1} \right|^p \right] &\leq K N^{-\frac{p}2}
	\end{align*}
	where $K$ only depends on $\varepsilon$ (which is obtained by first conditioning on $\tilde{W}_{(j)}$ then averaging over it).
	Taking $p>3$ and using the union bound on $n$ events, the Markov inequality and the Borel Cantelli lemma, along with $1_{\mathcal A}\asto 1$ and $\max_j \{|r_j^2/N-1|\}\asto 0$, leads to
	\begin{align*}
		\max_{1\leq j\leq n} \left|\frac1N w_j^*\check{S}_{(j),\eta}^{-1}w_j - \frac1N\tr \check{S}_{(j),\eta}^{-1} \right| \asto 0.
	\end{align*}
	Using the same result and the fact that $\frac1N\tr A^*\check{S}_{(j),\eta}^{-1}A\leq K\varepsilon^{-1}/N$ for all large $n$ a.s., we also have
	\begin{align*}
		\max_{1\leq j\leq n} \left|\frac1N s_j^*A^*\check{S}_{(j),\eta}^{-1}As_j \right| \asto 0.
	\end{align*}
	Using both results and $|\frac1N s_j^*A^*\check{S}_{(j),\eta}^{-1}w_j|^2\leq \frac1N s_j^*A^*\check{S}_{(j),\eta}^{-1}As_j\frac1Nw_j^*\check{S}_{(j),\eta}^{-1}w_j$ (Cauchy-Schwarz inequality), we finally get  
	\begin{align*}
		\max_{1\leq j\leq n} \left|\frac1N s_j^*A^*\check{S}_{(j),\eta}^{-1}w_j \right| \asto 0.
	\end{align*}
	All this then ensures that
	\begin{align*}
		\max_{1\leq j\leq n,\tau_j\leq M} \left| \frac1Ny_j^* \check{S}_{(j),\eta}^{-1}y_j - \tau_j \frac1N\tr \check{S}_{(j),\eta}^{-1} \right|&\asto 0 \\
		\max_{1\leq j\leq n,\tau_j> M} \left| \frac1{\tau_j} \frac1Ny_j^* \check{S}_{(j),\eta}^{-1}y_j - \frac1N\tr \check{S}_{(j),\eta}^{-1} \right| &\asto 0.
	\end{align*}

	Since $A$ has rank at most $L$, $\check{S}_{(j),\eta}$ is an at most rank-$2L+1$ perturbation of $\frac1n WT_{\eta} V_{\eta} W^*$, i.e., the matrix obtained for $p_1=\ldots=p_L=0$, by an additive symmetric matrix. A $(2L+1)$-fold application of the rank-one perturbation lemma \cite[Lemma~2.6]{SIL95} along with the facts that $\Vert W-\tilde{W}\Vert \asto 0$ and that all eigenvalues of the matrices involved are uniformly away from zero almost surely then ensures that 
\begin{align*}
	\max_{1\leq j\leq n} \left|\frac1N\tr \check{S}_{(j),\eta}^{-1} - \frac1N\tr \left( \frac1n WT_{\eta} V_{\eta} W^* \right)^{-1} \right| \asto 0.
\end{align*}
But now, recalling \cite[Remark~2]{COU13b}, $\frac1N\tr \left( \frac1n WT_{\eta} V_{\eta} W^* \right)^{-1}\asto \gamma^\eta$. Putting these results together finally leads to the requested result
\begin{align*}
	\max_{1\leq j\leq n,\tau_j\leq M} \left| \frac1Ny_j^* \check{S}_{(j),\eta}^{-1}y_j - \tau_j \gamma^\eta \right| &\asto 0 \\
	\max_{1\leq j\leq n,\tau_j>M} \left| \frac1{\tau_j} \frac1Ny_j^* \check{S}_{(j),\eta}^{-1}y_j - \gamma^\eta \right| &\asto 0.
\end{align*}
\end{proof}

Note that the proof only exploits the boundedness away from zero of the various matrices involved and not their bounded spectral norm. Therefore, with the same derivations, we also generalize \cite[Lemma~2]{COU13b} as follows.
\begin{lemma}
	\label{le:lemma2COU13b}
	For every $M>0$, we have
	\begin{align*}
		\max_{\substack{1\leq j\leq n \\ \tau_j\leq M}} \left| \frac1Ny_j^* \left( \frac1n\sum_{i\neq j} v\left( \tau_i\gamma^\eta\right) y_iy_i^* \right)^{-1}y_j - \tau_j\gamma \right| &\asto 0 \nonumber \\
		\max_{\substack{1\leq j\leq n \\ \tau_j > M}} \left| \frac1{\tau_j}\frac1Ny_j^* \left( \frac1n\sum_{i\neq j} v\left( \tau_i\gamma^\eta\right) y_iy_i^* \right)^{-1}y_j - \gamma \right| &\asto 0.
	\end{align*}
\end{lemma}

Define now $d_i=\frac1{\tau_i}\frac1Ny_i^*\hat{C}_{(i)}^{-1}y_i$ with $\hat{C}_{(i)}=\hat{C}_N-\frac1nu(\frac1Ny_i^*\hat{C}_N^{-1}y_i)y_iy_i^*$. Then \cite[Lemma~3]{COU13b} remains valid and reads
\begin{lemma}
	\label{le:lemma3COU13b}
There exists $d_+>d_->0$ such that, for all large $n$ a.s.
\begin{align*}
	d_i < \liminf_n \min_{1\leq i\leq n} d_i \leq \limsup_n\max_{1\leq i\leq n} d_i < d_+.
\end{align*}
\end{lemma}
\begin{proof}
	Taking $m>0$ small enough and denoting $d_{\rm max}=\max_j d_j$, Equation \cite[(14)]{COU13b} becomes here
	\begin{align*}
		\hat{C}_{(j)} \succeq mv(md_{\rm max})\frac1n\sum_{\substack{i\neq j\\ \tau_i\geq m}} \frac1{\tau_i}\left( As_is_i^*A^*+2\sqrt{\tau_i}\Re\left[w_i^*As_i\right] + \tau_i w_iw_i^*\right)
	\end{align*}
	so that, taking $j$ such that $d_j=d_{\rm max}$,
	\begin{align*}
		d_{\rm max} \leq \frac1{mv(md_{\rm max})} \frac1{\tau_j}\frac1N y_j^* \left( \frac1n\sum_{\substack{i\neq j\\ \tau_i\geq m}} \frac{As_is_i^*A^*+2\sqrt{\tau_i}\Re\left[w_i^*As_i\right] + \tau_i w_iw_i^*}{\tau_i} \right)^{-1}y_j.
	\end{align*}
	If $\liminf_n \tau_j>0$ (with $j$ always defined to be such that $d_j=d_{\rm max}$), with the same arguments as in the proof of Lemma~\ref{le:remark2COU13b} (here the boundedness from above of the $\tau_i$ is irrelevant) and recalling \cite[Lemma~6]{COU13b}, the right-hand side term can be bounded by $(mv_c(md_{\rm max}) (1-c) )^{-1}(1+\varepsilon)$ for arbitrarily small $\varepsilon>0$ by taking $m$ small enough and $n$ large enough. From there the proof of \cite[Lemma~3]{COU13b} for the boundedness of $d_{\rm max}$ remains valid. If instead $\liminf_n \tau_j=0$, we restrict ourselves to a subsequence over which $\tau_j\to 0$. Multiplying both sides of the equation above by $\tau_j$, we get by a similar result as Lemma~\ref{le:remark2COU13b} that $\tau_jd_{\rm max}$ can be bounded by $\tau_j( mv_c(md_{\rm max}) (1-c) )^{-1}(1+\varepsilon)$ for arbitrarily small $\varepsilon>0$ (again taking $m$ small and $n$ large), and the result unfolds again.

	To obtain the lower bound, in the proof of \cite[Lemma~3]{COU13b}, denoting $d_{\rm min}=\min_j d_j$, one needs now write
	\begin{align*}
		\hat{C}_{(j)} &\preceq Mv(Md_{\rm min}) \frac1n \sum_{\substack{i\neq j \\ m\leq \tau_i\leq M}} \frac1{\tau_i}\left( As_is_i^*A^*+2\sqrt{\tau_i}\Re\left[w_i^*As_i\right] + \tau_i w_iw_i^*\right) \nonumber \\
		& + v(0) \frac1n \sum_{\substack{i\neq j\\ \tau_i\in \RR\setminus [m,M]}} \left( As_is_i^*A^*+2\sqrt{\tau_i}\Re\left[w_i^*As_i\right] + \tau_i w_iw_i^*\right).
	\end{align*}
	The controls established for the upper bound on $d_{\rm max}$ can be similarly used here for $d_{\rm min}$ and the proof of \cite[Lemma~3]{COU13b} for $d_{\rm min}$ unfolds then similarly.
\end{proof}

Equipped with these lemmas, the proof of Theorem~\ref{th:1} unfolds similar to the proof of \cite[Theorem~2]{COU13b} but for a particular care to be taken for terms involving $\tau_j^{-1}y_j$ which need to be controlled if $\liminf_n \tau_j=0$. But this is easily performed as previously by either using approximations of $d_j$ or of $\tau_j d_j$ depending on whether $\liminf_n \tau_j>0$ or $\liminf_n \tau_j=0$, respectively. Assumption~\ref{ass:u}, which reproduces the assumptions of \citep{COU13b} are precisely used here. In particular, by the end of the proof, we obtain similar to \citep{COU13b} the important convergence
\begin{align}
	\label{eq:dilimit}
	\max_{1\leq j\leq n,\tau_j<M} \left| \tau_j d_j - \tau_j \gamma \right| &\asto 0 \nonumber \\
	\max_{1\leq j\leq n,\tau_j\geq M} \left| d_j - \gamma \right| &\asto 0
\end{align}
from which Theorem~\ref{th:1} easily unfolds.

\subsection{Eigenvalues of $\hat{C}_N$ and power estimation}

From Theorem~\ref{th:1}, $\Vert \hat{C}_N-\hat{S}_N\Vert\asto 0$ so that in particular $\max_{1\leq i\leq n}|\lambda_i(\hat{C}_N)-\lambda_i(\hat{S}_N)|\asto 0$. This means that it suffices to study the individual eigenvalues of $\hat{S}_N$ in order to study the individual eigenvalues of $\hat{C}_N$. In particular, from the results of Section~\ref{sec:localization}, we have that, for any small $\varepsilon>0$, $\hat{C}_N$ has asymptotically no eigenvalue in $[0,\varepsilon]$ almost surely, that $\hat{\lambda}_{|\mathcal L|+i}<S^++\varepsilon$ for all large $n$ a.s. for each $i\in \{1,\ldots,N-|\mathcal L|\}$ and that $\hat\lambda_i\asto \Lambda_i>S^+$ for each $i\in\mathcal L$, where $\Lambda_i$ is as in the statement of Theorem~\ref{th:2}, Item 0. Along with the continuity of $\delta$ and $\delta( (S^+,\infty) )\subset (-(\tau_+v_c(\tau_+\gamma)),0)$, we then get Theorem~\ref{th:2}, Item 1.

\subsection{Localization function estimation}
Let $a,b\in\CC^N$ be two vectors of unit norm. Then, from the first part of Theorem~\ref{th:2} and from Cauchy's integral formula, for any $k\in\mathcal L$ and for all large $N$ a.s.,
\begin{align}
	\label{eq:Cauchy}
	\sum_{\substack{1\leq i\leq L \\ p_i=p_\ell} } a^*\hat{u}_i\hat{u}_i^*b &= -\frac1{2\pi \imath} \oint_{\mathcal I_\ell} a^*\left( \hat{C}_N - zI_N \right)^{-1}b dz
\end{align}
for $\mathcal I_\ell$ defined as above as a positively oriented contour around a sufficiently small neighborhood of $\Lambda_\ell$, where $\Lambda_\ell$ is the unique positive solution of the equation in $\Lambda$ \eqref{eq:Lambda} when $p_k=p_\ell$. Using $\Vert \hat{C}_N-\hat{S}_N\Vert\asto 0$ along with the uniform boundedness of $\Vert (\hat{S}_N-zI_N)^{-1}\Vert$ and $\Vert (\hat{C}_N-zI_N)^{-1}\Vert$ on $\mathcal I_\ell$ (for all $n$ large), we then have
\begin{align*}
	\sum_{\substack{1\leq i\leq L \\ p_i=p_\ell} } a^*\hat{u}_i\hat{u}_i^*b + \frac1{2\pi \imath} \oint_{\mathcal I_\ell} a^*\left( \hat{S}_N - zI_N \right)^{-1}b dz &\asto 0
\end{align*}
so that it suffices to determine the second left-hand side expression.

Let us develop the term $a^*( \hat{S}_N - zI_N )^{-1}b$. Proceeding similar to Section~\ref{sec:localization}, we find
\begin{align*}
	a^*\left( \hat{S}_N - zI_N \right)^{-1}b &= a^*\left( \hat{S}_N^\circ - zI_N + \Gamma \right)^{-1}b
\end{align*}
with $\Gamma$ defined in \eqref{eq:Gamma}. Using Woodbury's identity $(A+BCB^*)^{-1}=A^{-1}-A^{-1}B(C^{-1}+B^*A^{-1}B)^{-1}B^*A^{-1}$ for invertible $A,B$, this becomes, with the same notations as in the previous paragraph,
\begin{align}
	\label{eq:aQb_develop}
	a^*\left( \hat{S}_N - zI_N \right)^{-1}b &= a^*Q^\circ_z b - a^*Q^\circ_z G\left(H^{-1}+G^*Q_z^\circ G\right)^{-1}G^*Q_z^\circ b
\end{align}
where
\begin{align*}
	G &= \begin{bmatrix} U\Omega^{\frac12} & \frac1n\tilde{W}T^{\frac12}VS^*\bar{U}\Omega^{\frac12} \end{bmatrix} \\
	H &=\begin{bmatrix} \Omega^{\frac12}\bar{U}^*\frac1n\tilde{W}V\tilde{W}^*\bar{U}\Omega^{\frac12} & I_L \\ I_L & 0 \end{bmatrix}.
\end{align*}
The matrix $H$ is clearly invertible and we then find, using Lemma~\ref{le:1} and Lemma~\ref{le:2} that, uniformly on $z$ in a small neighborhood of $\Lambda_\ell$,	
\begin{align*}
	\left\Vert H^{-1} - \begin{bmatrix} 0 & I_L \\ I_L & -\Omega \frac1n\tr V \end{bmatrix} \right\Vert \asto 0
\end{align*}
so that, again by Lemma~\ref{le:1} and Lemma~\ref{le:2},
\begin{align}
	\label{eq:H+GQG}
	\left\Vert H^{-1}+G^*Q_z^\circ G - \begin{bmatrix} \Omega \frac{\delta(z)}c & I_L \\ I_L & -\Omega \frac1n\tr V\Psi_z^\circ \end{bmatrix} \right\Vert &\asto 0.
\end{align}
To ensure that $H^{-1}+G^*Q_z^\circ G$ is invertible for $z\in \mathcal I_\ell$, let us study the determinant of the rightmost matrix. We have easily
\begin{align*}
	\det \left( \begin{bmatrix} \Omega \frac{\delta(z)}c & I_L \\ I_L & -\Omega \frac1n\tr V\Psi_z^\circ \end{bmatrix} \right) &= \det \left( - \Omega^2 \frac{\delta(z)}c \frac1n\tr V\Psi_z^\circ - I_L \right).
\end{align*}
From the discussion around \eqref{eq:Lambda}, the right-hand side term cancels exactly once in a neighborhood of $z=\Lambda_k$ for each $k\in\mathcal L$. Now, for $z\in \CC\setminus \RR$, it is easily seen that it has non-zero imaginary part. Therefore, since the convergence \eqref{eq:H+GQG} is uniform on a small neighborhood of $\Lambda_\ell$, for all large $n$ a.s., the determinant of $H^{-1}+G^*Q_z^\circ G$ is uniformly away from zero on $\mathcal I_\ell$ (up to taking $n$ larger).
We can then freely take inverses in \eqref{eq:H+GQG} and have, uniformly on $\mathcal I_\ell$,
\begin{align*}
	\left\Vert (H^{-1} + G^*Q_z^\circ G)^{-1} - \begin{bmatrix} \Omega \frac{\delta(z)}c & I_L \\ I_L & -\Omega \frac1n\tr V\Psi_z^\circ \end{bmatrix}^{-1} \right\Vert &\asto 0.
\end{align*}
To compute the inverse of the rightmost matrix, it is convenient to write
\begin{align*}
	\begin{bmatrix} \Omega \frac{\delta(z)}c & I_L \\ I_L & \Omega \frac1n\tr V\Psi_z^\circ \end{bmatrix} &= P \left\{\begin{bmatrix} \sigma_k\frac{\delta(z)}c & 1 \\ 1 & -\sigma_k\frac1n\tr V\Psi_z^\circ \end{bmatrix} \right\}_{k=1}^L P^*
\end{align*}
where $\{A_k\}_{k=1}^L$ is a block-diagonal matrix with diagonal blocks $A_1,\ldots,A_L$ in this order, and where $P\in\CC^{2L\times 2L}$ is the symmetric permutation matrix with $[P]_{ij}={\bm\delta}_{j-(L+i/2)}$ for even $i\leq L$ and $[P]_{ij}={\bm\delta}_{j-(i+1)/2}$ for odd $i\leq L$. With this notation, we have
\begin{align*}
	\begin{bmatrix} \Omega \frac{\delta(z)}c & I_L \\ I_L & \Omega \frac1n\tr V\Psi_z^\circ \end{bmatrix}^{-1} &= P \left\{ \frac{-1}{\frac{\delta(z)}c\sigma_k^2 \frac1n\tr V\Psi_z^\circ + 1} \begin{bmatrix} -\sigma_k\frac1n\tr V\Psi_z^\circ & -1 \\ -1 & \sigma_k \frac{\delta(z)}c \end{bmatrix} \right\}_{k=1}^L P^*.
\end{align*}

Denoting $U=[u_1,\ldots,u_L]\in\CC^{N\times L}$ and $\bar{U}=[\bar{u}_1,\ldots,\bar{u_L}]\in\CC^{L\times L}$, we have
\begin{align*}
	GP &= \begin{bmatrix} \sqrt{\sigma_1} u_1 & \sqrt{\sigma_1}\frac1n\tilde{W}T^{\frac12}VS^*\bar{u}_1 & \cdots & \sqrt{\sigma_L} u_L & \sqrt{\sigma_L}\frac1n\tilde{W}T^{\frac12}VS^*\bar{u}_L \end{bmatrix}.
\end{align*}

From this remark, using again Lemma~\ref{le:1} and Lemma~\ref{le:2}, we finally have
\begin{align*}
	\sup_{z\in\mathcal I_\ell}\left| a^*Q_z^\circ G(H^{-1} +G^*Q_z^\circ G)^{-1}G^* Q_z^\circ b - \sum_{k=1}^L a^*u_ku_k^*b \frac{\frac{\delta(z)^2}{c^2}\sigma_k^2 \frac1n\tr V\Psi_z^\circ}{\frac{\delta(z)}c\sigma_k^2 \frac1n\tr V\Psi_z^\circ+1}\right| \asto 0.
\end{align*}

Putting things together, using the results above which we recall are uniform on $\mathcal I_\ell$, and also using the fact that $Q_z^\circ$ has no pole in $\mathcal I_\ell$, we finally have 
\begin{align*}
	\sum_{\substack{1\leq i\leq L \\ p_i=p_\ell} } a^*\hat{u}_i\hat{u}_i^*b  - \sum_{k=1}^L \frac1{2\pi \imath} \oint_{\mathcal I_\ell} a^*u_ku_k^*b \frac{\frac{\delta(z)^2}{c^2}\sigma_k^2 \frac1n\tr V\Psi_z^\circ}{\frac{\delta(z)}c\sigma_k^2 \frac1n\tr V\Psi_z^\circ+1} dz \asto 0
\end{align*}
which, after taking the limits on the fraction in the integrand, gives
\begin{align*}
	\sum_{\substack{1\leq i\leq L \\ p_i=p_\ell} } a^*\hat{u}_i\hat{u}_i^*b  - \sum_{k=1}^L \frac1{2\pi \imath} \oint_{\mathcal I_\ell} a^*u_ku_k^*b \frac{\frac{\delta(z)^2}{c^2}p_k \int \frac{v(\tau \gamma)\nu(d\tau)}{1+\tau v(\tau \gamma)\delta(z)}}{\frac{\delta(z)}cp_k \int \frac{v(\tau \gamma)\nu(d\tau)}{1+\tau v(\tau \gamma)\delta(z)}+1} dz \asto 0
\end{align*}

For $z\in (S^+,\infty)$, we already saw that $\delta(z)$ is negative while $\int \frac{v(\tau \gamma)\nu(d\tau)}{1+\tau v(\tau \gamma)\delta(z)}$ is positive. For $z$ non real, both quantities are non real, and therefore do no have poles in $\mathcal I_\ell$. The only pole is then obtained for $\frac{\delta(z)}cp_k \int \frac{v(\tau \gamma)\nu(d\tau)}{1+\tau v(\tau \gamma)\delta(z)}+1=0$, that is for $z=\Lambda_\ell$ as defined in the previous section. Using l'Hospital rule, the residue of the right complex integral is then evaluated to be
\begin{align}
	{\rm Res}(\Lambda_{\ell}) &= \lim_{z\to \Lambda_{\ell}} (z-\Lambda_{\ell}) a^*\Pi_\ell b\frac{\delta(z)}c \frac{\int \frac{p_{\ell} v(t\gamma)}{1+tv(t\gamma)\delta(z)}\nu(dt)}{\int \frac{p_{\ell} v(t\gamma)}{1+tv(t\gamma)\delta(z)}\nu(dt) + \frac{c}{\delta(z)}} \nonumber \\
	&= \lim_{z\to \Lambda_{\ell}} a^*\Pi_\ell b \frac{\frac{\delta(z)}c \int \frac{v(t\gamma)p_{\ell}}{1+tv(t\gamma)\delta(z)}\nu(dt)+ (z-\Lambda_{\ell}) \frac{d}{dz} \left( \frac{\delta(z)}c \int \frac{p_{\ell} v(t\gamma)}{1+tv(t\gamma)\delta(z)}\nu(dt)\right)}{-c\frac{\delta'(z)}{\delta(z)^2}- \int \frac{t v(t\gamma)^2 p_{\ell} \delta'(z)}{(1+tv(t\gamma)\delta(z))^2} \nu(dt)} \nonumber \\
	\label{eq:res}
	&= a^*\Pi_\ell b \left( c\frac{\delta'(\Lambda_{\ell})}{\delta(\Lambda_{\ell})^2} + p_{\ell} \delta'(\Lambda_{\ell})\int \frac{\tau v_c(\tau\gamma)^2\nu(d\tau)}{(1+\tau v_c(\tau\gamma)\delta(\Lambda_{\ell}))^2} \right)^{-1}
\end{align}
where $\Pi_\ell\triangleq \sum_{i,p_i=p_\ell} u_iu_i^*$ and the last equality uses $\frac{\delta(\Lambda_{\ell})}cp_\ell \int \frac{v(\tau \gamma)\nu(d\tau)}{1+\tau v(\tau \gamma)\delta(\Lambda_{\ell})}=-1$. 
Recall now that
\begin{align*}
	\delta(\Lambda_{\ell}) &= c \left(-\Lambda_{\ell} + \int \frac{\tau v_c(\tau\gamma)}{1+\delta(\Lambda_{\ell})\tau v_c(\tau \gamma)} \nu(d\tau)\right)^{-1}
\end{align*}
from which
\begin{align*}
	\delta'(\Lambda_{\ell}) &= \frac{\delta(\Lambda_{\ell})^2}c \left( 1 - \frac{\delta(\Lambda_{\ell})^2}c\int \frac{\tau^2v_c(\tau\gamma)^2}{(1+\delta(\Lambda_{\ell})\tau v_c(\tau\gamma))^2}\nu(d\tau) \right)^{-1} > 0.
\end{align*}

From the expression of $p_{\ell}$ in the previous paragraph and these values, we then further find
\begin{align*}
	{\rm Res}(\Lambda_{\ell})&= a^*\Pi_\ell b \left(1- \frac{\int \frac{\delta(\Lambda_{\ell})\tau v_c(\tau\gamma)^2\nu(d\tau)}{(1+\delta(\Lambda_{\ell})\tau v_c(\tau\gamma))^2}}{\int \frac{v_c(\tau\gamma)\nu(d\tau)}{1+\delta(\Lambda_{\ell})\tau v_c(\tau\gamma)}}\right)^{-1} \left({1- \frac{\delta(\Lambda_{\ell})^2}c \int \frac{t^2v_c(\tau\gamma)^2\nu(d\tau)}{(1+\delta(\Lambda_{\ell})\tau v_c(\tau\gamma))^2}}\right) \\
	&=  a^*\Pi_\ell b\frac{\int \frac{v_c(\tau\gamma)\nu(d\tau)}{1+\delta(\Lambda_{\ell})\tau v_c(\tau\gamma)} \left( 1 - \frac{\delta(\Lambda_{\ell})^2}c \int \frac{t^2v_c(\tau\gamma)^2\nu(d\tau)}{(1+\delta(\Lambda_{\ell})\tau v_c(\tau\gamma))^2} \right)}{\int \frac{v_c(\tau\gamma)\nu(d\tau)}{(1+\delta(\Lambda_{\ell})\tau v_c(\tau\gamma))^2}}.
\end{align*}
Inverting the relation 
\begin{align*}
	\sum_{\substack{1\leq i\leq L \\ p_i=p_\ell} } a^*\hat{u}_i\hat{u}_i^*b - a^*\Pi_\ell b\frac{\int \frac{v_c(\tau\gamma)\nu(d\tau)}{1+\delta(\Lambda_{\ell})\tau v_c(\tau\gamma)} \left( 1 - \frac{\delta(\Lambda_{\ell})^2}c \int \frac{t^2v_c(\tau\gamma)^2\nu(d\tau)}{(1+\delta(\Lambda_{\ell})\tau v_c(\tau\gamma))^2} \right)}{\int \frac{v_c(\tau\gamma)\nu(d\tau)}{(1+\delta(\Lambda_{\ell})\tau v_c(\tau\gamma))^2}} &\asto 0
\end{align*}
and using $\hat{\lambda}_{\ell}\asto \Lambda_{\ell}$ for all $\ell\in\mathcal L$ then completes the proof.

\subsection{Empirical estimators}

To prove Theorem~\ref{th:3}, one needs to ensure that the empirical estimators introduced in the statement of the theorem are consistent with the estimators introduced in Theorem~\ref{th:2}. 

Note first that $\gamma-\hat{\gamma}_n\asto 0$ is a consequence of \eqref{eq:dilimit}. Indeed, letting $M>0$, from \eqref{eq:dilimit}, 
\begin{align*}
	\frac1n\sum_{\tau_j<M} \tau_jd_j - \gamma\frac1n\sum_{\tau_j<M} \tau_j \asto 0.
\end{align*}
Still from \eqref{eq:dilimit}, we also have, a.s.
\begin{align*}
	\frac1n\sum_{\tau_j\geq M} \tau_jd_j - \gamma\frac1n\sum_{\tau_j\geq M} \tau_j = o\left( \frac1n\sum_{\tau_j\geq M}\tau_j \right).
\end{align*}
But $\frac1n\sum_{\tau_j\geq M} \tau_j\asto \int_{(M,\infty)} t \nu(dt)\leq 1$ (say $M$ is a continuity point of $\nu$). Also, $\frac1n\sum_j \tau_j\asto 1$. Putting the results together then gives $\gamma-\hat{\gamma}_n\asto 0$. From this, we now get, again with \eqref{eq:dilimit},
\begin{align*}
	\max_{1\leq j\leq n,\tau_j\leq M} \left| \frac{\tau_j d_j}{\hat{\gamma}_n} - \tau_j \right| &\asto 0 \\
	\max_{1\leq j\leq n,\tau_j>M} \left| \frac{d_j}{\hat{\gamma}_n} - 1 \right| &\asto 0
\end{align*}
which is $\max_{\tau_j\leq M}|\tau_j-\hat{\tau}_j|\asto 0$ and $\max_{\tau_j>M}|\tau_j^{-1}\hat{\tau}_j-1|\asto 0$, as desired.

We now need to prove that $\hat{\delta}(x)-\delta(x)\asto 0$ uniformly on any bounded set of $(S^++\varepsilon,\infty)$. For this, recall first that both $\hat{\delta}$ and $\delta$ are Stieltjes transforms of distributions with support contained in $[0,S^+]$ and, as such, are analytic in $(S^++\varepsilon,\infty)$ and uniformly bounded in any compact of $(S^++\varepsilon,\infty)$. Taking the difference and denoting $\hat{\nu}_n=\frac1n\sum_{i=1}^n{\bm\delta}_{\hat{\tau}_i}$, we have
\begin{align*}
	&\hat{\delta}(x)-\delta(x) \nonumber \\
	&= \left(1-\frac{c}{c_n}\right)\hat{\delta}(x) + \frac{ \hat{\delta}(x)\delta(x)}{c_n} \left( \int \frac{t v_c(t\gamma)\nu(dt)}{1+\delta(x)tv_c(t\gamma)} - \int \frac{t v_c(t\hat\gamma_n)\hat\nu_n(dt)}{1+\hat\delta(x)tv_c(t\hat\gamma_n)}  \right) \\
	&= \left(1-\frac{c}{c_n}\right)\hat{\delta}(x) + \frac{ \hat{\delta}(x)\delta(x)}{c_n} \left((\hat{\delta}(x)-\delta(x))\int \frac{t^2v_c(t\gamma)v_c(t\hat\gamma_n)\nu(dt)}{(1+\delta(x)tv_c(t\gamma))(1+\hat\delta(x)tv_c(t\hat\gamma_n))} \right. \nonumber \\
	&\left. + \int \frac{t(v_c(t\gamma)-v_c(t\hat\gamma_n))\nu(dt)}{(1+\delta(x)tv_c(t\gamma))(1+\hat\delta(x)tv_c(t\hat\gamma_n))} + \int \frac{t v_c(t\hat\gamma_n)(\hat\nu_n(dt)-\nu(dt))}{1+\hat\delta(x)tv_c(t\hat\gamma_n)} \right).
\end{align*}
From uniform boundedness of $tv_c(t\hat\gamma_n)$ and $tv_c(t\gamma)$, and $\hat\nu_n( (t,M) )\asto \nu( (t,M) )$ weakly and $\hat\gamma_n\asto \gamma$, it is easily seen that the last two integrals on the right-hand side can be made arbitrarily small (e.g., by isolating $\tau_i\leq M$ and $\tau_i>M$ and letting $M$ large enough in the previous convergence). Also, the first integral on the right hand side is clearly bounded. Gathering the terms $\hat{\delta}(x)-\delta(x)$ on the left-hand side and taking $x$ large enough so to ensure $\hat\delta(x)\delta(x)$ is uniformly smaller than one (recall that their limit is zero as $x\to\infty$), we finally get that $\hat\delta(x)-\delta(x)$ can be made arbitrarily small. This is valid for any given large $x$ and therefore on some sequence $\{x^{(i)}\}$ of $(S^++\varepsilon,\infty)$ having an accumulation point, $\hat\delta(x^{(i)})\delta(x^{(i)})\asto 0$. Since $\hat\delta(x)-\delta(x)$ is complex analytic in $(S^++\varepsilon,\infty)$, by Vitali's convergence theorem, we therefore get that the convergence is uniform over any bounded set of $(S^++\varepsilon,\infty)$, which is what we wanted.

Since, for $i\in\mathcal L$ and for some $\varepsilon,M>0$, $\hat{\lambda}_i\in [S^++\varepsilon,M]$ for all large $n$ a.s., we therefore have that $\hat\delta(\hat\lambda_i)-\delta(\lambda_i)\asto 0$ for each $i\in\mathcal L$. Using all these convergence results, we then obtain, with the same line of arguments the asymptotic consistence between the estimates in Item~1.\@ and Item~2.\@ of both Theorems~\ref{th:2} and \ref{th:3}. This concludes the proof of Theorem~\ref{th:3}.

\subsection{Proof of Corollary~\ref{co:RGMUSIC}}
	We are here in the same setting as \cite[Theorem~3]{HAC13b}, only for our improved model. The proof is the same as in \citep{HAC13b} and relies on showing the uniform convergence of $\hat{\eta}_{\rm RG}(\theta)-\eta(\theta)$ across $\theta$, from which the result unfolds. In our setting, the point-wise convergence easily follows from Items~3.\@ in both Theorem~\ref{th:2} and Theorem~\ref{th:3}. Uniform convergence then hinges on a regular discretization of the set $[0,2\pi)$ into $N^2$ subsets and on (i) a Lipschitz control of the differences $\hat{\eta}_{\rm RG}(\theta)-\hat{\eta}_{\rm RG}(\theta')$ for $|\theta-\theta'|=O(N^{-2})$ and (ii) a joint convergence of $\hat{\eta}_{\rm RG}(\theta)-\eta(\theta)$ over the $N^2+1$ edges of the subsets. Point (i) uses the defining properties of $a(\theta)$ from Assumption~\ref{ass:arrayprocessing} similar to \citep{HAC13b}, while Point (ii) is obtained thanks to a classical union bound on $N^2$ events, the validity of which follows from considering sufficiently high order moment bounds on the vanishing random quantities involved in $\hat{\eta}_{\rm RG}(\theta)-\eta(\theta)$. In our setting, the latter moment bounds are obtained by selecting $p$ large enough in Lemma~\ref{le:1} of Section~\ref{sec:proof} (in a similar fashion as is performed for the technical proof that $\min_{j}\lambda_N(\hat{S}_{(j)})>\varepsilon$ for all large $n$ a.s. in Section~\ref{sec:proof}). It is easily seen that, this being ensured, the proof of Corollary~\ref{co:RGMUSIC} unfolds similar to that of \cite[Theorem~3]{HAC13b}, which as a consequence we do not further detail.

\bibliographystyle{elsarticle-harv}
\bibliography{/home/romano/Documents/PhD/phd-group/papers/rcouillet/tutorial_RMT/book_final/IEEEabrv.bib,/home/romano/Documents/PhD/phd-group/papers/rcouillet/tutorial_RMT/book_final/IEEEconf.bib,/home/romano/Documents/PhD/phd-group/papers/rcouillet/tutorial_RMT/book_final/tutorial_RMT.bib}

\end{document}